\documentclass[10pt]{article}

\usepackage{amsmath}
\usepackage{amsfonts}
\usepackage{amsthm}
\usepackage[english]{babel}
\usepackage{graphicx}
\usepackage[all]{xy}

\newtheorem{theorem}{Theorem}
\newtheorem{lemma}{Lemma}
\newtheorem{definition}{Definition}
\newtheorem{corollary}{Corollary}

\newtheorem{remark}{Remark}

\newcommand{\mR}{\mathbb{R}}
\newcommand{\mC}{\mathbb{C}}
\newcommand{\mN}{\mathbb{N}}

\newcommand{\mZ}{\mathbb{Z}}
\newcommand{\mS}{\mathbb{S}}

\newcommand{\mF}{\mathbb{F}}

\newcommand{\mg}{\mathfrak{g}}
\newcommand{\cD}{\mathcal{D}}

\newcommand{\cC}{\mathcal{C}}

\newcommand{\cU}{\mathcal{U}}

\newcommand{\mn}{\mathfrak{n}}

\newcommand{\osp}{\mathfrak{osp}(m|2n)}

\begin{document}
\title{On a class of tensor product representations for the orthosymplectic superalgebra}

\author{Kevin\ Coulembier\thanks{Ph.D. Fellow of the Research Foundation - Flanders (FWO), E-mail: {\tt Coulembier@cage.ugent.be}}}

\date{\small{Department of Mathematical Analysis}\\
\small{Faculty of Engineering -- Ghent University\\ Krijgslaan 281, 9000 Gent,
Belgium}\\
}

\maketitle

\begin{abstract}
The spinor representations for $\osp$ are introduced. These generalize the spinors for $\mathfrak{so}(m)$ and the symplectic spinors for $\mathfrak{sp}(2n)$ and correspond to representations of the supergroup with supergroup pair $(Spin(m)\times Mp(2n),\osp)$. These spinor spaces are proved to be uniquely characterized as the completely pointed $\osp$-modules. The main aim is to study the tensor product of these representations with irreducible finite dimensional $\osp$-modules. Therefore a criterion for complete reducibility of tensor product representations of semisimple Lie superalgebras is derived. Finally the decomposition into irreducible $\osp$-representations of the tensor product of the super spinor space with an extensive class of such representations is calculated and also cases where the tensor product is not completely reducible are studied. 
\end{abstract}



\section{Introduction}


Tensor product representations of (super)groups play an important role in theoretical physics. See e.g. \cite{MR1632811, BG, MR1297597, MR2028498, MR0482879, MR1738448, MR2286881} for important results on such tensor products. In this paper we investigate the complete reducibility of tensor products of irreducible highest weight representations for Lie superalgebras and in particular the complete decomposition of a certain class of tensor product representations for the orthosymplectic superalgebra $\osp$. These generalize the tensor product representations of the spinor representations $\mS_m$ of $Spin(m)$ with finite dimensional $SO(m)$-representations. We introduce the spinor representation for $\osp$, $\mS_{m|2n}$, which is related to the oscillator realization of $\osp$ in \cite{MR1037401} or chapter 29 in \cite{MR1773773}. This space $\mS_{m|2n}$ corresponds to a representation of the Lie supergroup with supergroup pair $(Spin(m)\times Mp(2n),\osp)$ and generalizes both the $\mathfrak{so}(m)$-spinors and the symplectic spinors for $\mathfrak{sp}(2n)$ introduced in \cite{MR0400304}. As in the case of $\mathfrak{sp}(2n)$, the spinor spaces are the only completely pointed highest weight $\osp$-modules. This representation also appears in \cite{OSpHarm} in the context of Howe dualities for Lie superalgebras and as is proven in \cite{Joseph}, the annihilator ideal in the universal enveloping algebra $\cU(\mathfrak{osp}(m|2n))$ of this representation is a Joseph-type ideal. Therefore this representation has a certain interpretation as a minimal representation.

One specific motivation to study the decomposition of such tensor products comes from the study of first order invariant differential operators on supermanifolds. For the ungraded case, see e.g. \cite{MR2782791, MR0482879, MR1738448, MR2286881, MR0223492}. In the unified construction of $Spin(m)$-invariant generalized Cauchy-Riemann operators of Stein and Weiss (see \cite{MR0223492}), the tensor product of the fundamental representation of $\mathfrak{so}(m)$ with other $\mathfrak{so}(m)$-representations needs to be calculated. The same decomposition of tensor products is needed in the construction of Fegan of conformally invariant differential operators in \cite{MR0482879}. The case of the spinor representation leads to the Dirac operator and higher spinor representations lead to Rarita-Schwinger (or higher spin)-type operators, see e.g. \cite{MR2782791}. These constructions have already been made in case $\mathfrak{so}(m)$ is replaced by $\mathfrak{sp}(2n)$ in \cite{MR2286881}, using the symplectic (higher) spinor spaces from \cite{MR1615943}. To generalize the Dirac operator to superspace the tensor product of the spinor space $\mS_{m|2n}$ with the fundamental $\osp$-representation $L^{m|2n}_{\delta_1}$ is needed, which is a specific case of the results in this paper. In order to explicitly realize the higher spin operators on $\mR^m$ from \cite{MR0482879, MR0223492}, such as in \cite{MR2782791}, the tensor product of spaces of (simplicial) harmonics with spinors needs to be calculated. Based on the representation-theoretical results on spherical harmonics in superspace in \cite{OSpHarm, MR2395482} we find that the corresponding tensor products for $\osp$, which are needed to construct graded version of higher spinor operators, again correspond to special cases of the tensor product investigated in this paper. These tensor products are also necessary to study the $\osp$-representation structure of the kernel of the aforementioned super Dirac operator. 

Another reason to study these tensor products comes from the classification of representations with bounded dimensions of the weight spaces, see \cite{MR1297597, MR1615943, MR1013330}. In the non-graded case the tensor product of finite dimensional representations with completely pointed modules plays the essential role in the classification. As is shown in \cite{MR1615943}, the modules with bounded weight-multiplicities for $\mathfrak{sp}(2n)$ correspond exactly to the aforementioned symplectic higher spinor spaces, see \cite{MR2286881}. The higher spinor spaces for $\osp$, which are representations with bounded multiplicities, will appear in the tensor product decompositions in the current paper. 

In Section 9 of the paper \cite{MR2028498} the tensor product of certain infinite dimensional unitarizable $\osp$-representations was obtained from the theory of Howe duality. The paper \cite{MR1632811} is devoted to the tensor product of finite dimensional $\osp$-representations $\left(L^{m|2n}_{\delta_1}\right)^{\otimes k}$. There, the decomposition was studied based on the action of the Brauer algebra. Finite dimensional $\osp$-representations are never unitarizable (star-representations) except for $C(n)=\mathfrak{osp}(2,2n-2)$, see \cite{MR0424886}, so their tensor products are not necessarily completely reducible. In the current paper we investigate tensor products of a combination of these two types of $\osp$-representations. So we decompose the tensor product of an infinite dimensional unitarizable representation with finite dimensional representations into irreducible pieces. Again by lack of unitarizablility, it is not a priori known if the tensor products will be completely reducible. In Section \ref{criterion} we therefore derive a useful criterion for complete reducibility of tensor product representations of semisimple Lie superalgebras. This is based on insight into the structure of primitive vectors which can be obtained by using a notion which generalizes unitarizability. The tensor products of irreducible highest weight representations always satisfy this generalized unitarizability, which helps to determine their complete reducibility. Some of the tensor products we study do turn out to be not completely reducible. For these cases we obtain the complete decomposition series.

The paper is organized as follows. In Section \ref{sectionroot} we introduce our root system for $\osp$ which corresponds to the one in \cite{OSpHarm, MR2395482} and differs from the standard one in \cite{MR1773773, MR051963}. To make a distinction we use the notation $K^{m|2n}_\mu$ for the irreducible representation with highest weight $\mu$ in our root system and $L^{m|2n}_\lambda$ for the irreducible representation with highest weight $\lambda$ in the standard root system. All calculations and proofs will be performed in our choice but important results will be stated in both systems. Going from one system to the other uses the technique of odd reflections from \cite{MR1327543, MR1201236}, which will be explained in Section \ref{sectionrefl}. In section 3 we state the results of the spinor spaces and tensor products for the classical cases, these can be obtained from the results in \cite{MR1297597, MR1738448, MR1401053}. In Section \ref{criterion} we derive a criterion for complete reducibility of tensor products of semisimple Lie superalgebras. In Section \ref{sectionsuperspin}, we define the $\osp$-spinor spaces and show that they are uniquely determined as the completely pointed $\osp$-modules. Then the tensor products are studied. In Section \ref{sectionrestr} we find important restrictions on the possible primitive vectors. In Section \ref{secDec} we obtain the explicit decomposition for an extensive class of representations, but also investigate cases that are not completely reducible. Finally, in Section \ref{concl}, a brief overview of the main results is given and some logical directions for extensions of these results are discussed.

\section{Root systems for $\osp$}
\label{sectionroot}
In this section we introduce the orthosymplectic Lie superalgebra $\osp$ as the subalgebra of $\mathfrak{gl}(m|2n)$ which fixes the orthosymplectic metric. We always consider the complex algebra, so $\osp=\mathfrak{osp}(m|2n;\mC)$. In this section we will also introduce a non-standard choice of positive roots and simple roots and the corresponding Chevalley basis.

The orthosymplectic metric $g$ is given in block-matrix form by
\begin{eqnarray*}
g=\left( \begin{array}{cc}h&0\\  \vspace{-3.5mm} \\0&J \end{array} \right)&\mbox{with}& J=\left( \begin{array}{cc}0&I_n\\  \vspace{-3.5mm} \\-I_n&0\end{array} \right)
\end{eqnarray*}
where $h$ is
\begin{eqnarray}
\label{metricom}
\left( \begin{array}{cc}0&I_d\\  \vspace{-3.5mm} \\I_d&0 \end{array} \right)\quad\mbox{if}\quad m=2d &\mbox{and}& \left( \begin{array}{ccc}0&I_d&0\\  \vspace{-3.5mm} \\I_d&0&0\\  \vspace{-3.5mm} \\ 0&0&1 \end{array} \right)\quad\mbox{if}\quad m=2d+1.
\end{eqnarray}
The Lie superalgebra $\osp$ is given by the matrices $A\in\mC^{(m+2n)\times(m+2n)}=\mathfrak{gl}(m|2n;\mC)$ satisfying
\begin{eqnarray*}
A^{sT} g+gA=0&\mbox{with}& \left(\begin{array}{cc}a&b\\  \vspace{-3.5mm} \\c&d \end{array} \right)^{sT}=\left(\begin{array}{cc}a^T&-c^T\\  \vspace{-3.5mm} \\b^T&d^T \end{array} \right).
\end{eqnarray*}
The $\mZ_2$-gradation is induced by the embedding in $\mathfrak{gl}(m|2n)$, so matrices of the form $\left(\begin{array}{cc}a&0\\  \vspace{-3.5mm} \\0&d \end{array} \right)$ are homogeneous even elements and matrices of the form $\left(\begin{array}{cc}0&b\\  \vspace{-3.5mm} \\c&0 \end{array} \right)$ homogeneous odd elements. These two subspaces are denoted respectively by $\osp_0$ and $\osp_1$.

For a Lie superalgebra $\mathfrak{g}=\mathfrak{g}_0\oplus \mathfrak{g}_1$ with even part the underlying Lie algebra $\mathfrak{g}_0$ and odd part $\mathfrak{g}_1$, we use the notation $|X|=i\in \mZ_2$, if $X\in\mathfrak{g}_i$. We will always use the notation $[\cdot,\cdot]$ for the super Lie bracket. This super Lie bracket is super anti-symmetric, $[X,Y]=-(-1)^{|X||Y|}[Y,X]$ for $X,Y$ homogeneous, and satisfies that super-Jacobi identity, see e.g. \cite{MR1773773, MR051963}. This is modeled after the graded commutator on a general $\mZ_2$-graded algebra: $[X,Y]=XY-(-1)^{|X||Y|}YX$ for $X$ and $Y$ homogeneous elements of $\mathfrak{osp}(m|2n)\subset\mathfrak{gl}(m|2n)= \mC^{(m+2n)\times (m+2n)}$ and $XY$ seen as matrix multiplication.

A representation of a Lie superalgebra $\mathfrak{g}$ on a super vector space $V$ is a morphism of superalgebras between $\mathfrak{g}$ and the Lie superalgebra $End(V)$, see \cite{MR1773773}. Here $End(V)$ has a natural gradation induced by $V$ and the Lie superbracket is the graded commutator. The tensor product of two $\mathfrak{g}$-representations becomes a representation with action given by
\begin{eqnarray*}
X(a\otimes b)&=&Xa\otimes b+(-1)^{|X||a|}a\otimes Xb,
\end{eqnarray*}
for $X\in\mathfrak{g}$ homogeneous and $a$ and $b$ homogeneous vectors in the two representations.

The Cartan subalgebra generated by the diagonal matrices in $\osp$ will be denoted by $\mathfrak{h}$. The weights for $\osp$ (the space $\mathfrak{h}^\ast$) can be expressed in terms of $\epsilon_j$, $j=1,\cdots,d=\lfloor m/2\rfloor$ and $\delta_i$, $i=1,\cdots,n$. These are linear functionals on the space of diagonal matrices $\cD\subset\mC^{(m+2n)\times (m+2n)}$, defined by
\begin{eqnarray*}
\epsilon_j(E_{kk})=\delta_{k,j}&\mbox{for}&j=1,\cdots d\quad\mbox{and}\quad k=1,\cdots,m+2n\\
\delta_i(E_{kk})=\delta_{k,i+m}&\mbox{for}&i=1,\cdots n\quad\mbox{and}\quad k=1,\cdots,m+2n.
\end{eqnarray*} 
The matrix with notation $E_{pq}$ is the matrix satisfying $(E_{pq})_{st}=\delta_{sp}\delta_{tq}$ for all $1\le s,t\le m+2n$.

By restricting the definition from general diagonal matrices to $\mathfrak{h}$ we obtain elements of $\mathfrak{h}^\star$, which we also denote by $\epsilon_j$ and $\delta_i$. The (symmetric but non-definite) inner product on $\mathfrak{h}^\ast$ is given by $\langle\cdot,\cdot\rangle$
\begin{eqnarray}
\label{weightinprod}
\langle \epsilon_j,\epsilon_k\rangle=\frac{1}{2}\delta_{jk},\quad\langle \epsilon_j,\delta_i\rangle=0,\quad \langle \delta_i,\delta_l\rangle=-\frac{1}{2}\delta_{il}.
\end{eqnarray}

The fundamental weights of $\mathfrak{so}(m)$ and $\mathfrak{sp}(2n)$ are given in the following definition.
\begin{definition}
\label{classweights}
The fundamental weights of $\mathfrak{so}(2d+1)$ are given by $\omega_k=\sum_{j=1}^k\epsilon_j$ for $1\le k\le d-1$ and $\omega_d=\frac{1}{2}(\sum_{j=1}^d\epsilon_j)$. The fundamental weights of $\mathfrak{so}(2d)$ with $d>1$ are given by $\omega_k=\sum_{j=1}^k\epsilon_j$ for $k\le d-2$, $\omega_{d-1}=\frac{1}{2}(\sum_{j=1}^{d-1}\epsilon_j-\epsilon_d)$ and $\omega_d=\frac{1}{2}(\sum_{j=1}^d\epsilon_j)$. For $\mathfrak{sp}(2n)$ the fundamental weights are given by $\nu_l=\sum_{i=1}^l\delta_i$ for $1\le l\le n$. 
\end{definition}
We also use the notations $\omega_0=\nu_0=0$.

Next we introduce a non-standard Chevalley basis for $\osp$, this does not correspond to the distinguished basis of simple roots. We need to make a distinction between $m$ even and odd. 

\subsection{The case $\mathfrak{osp}(2d+1|2n)$}

The even roots of $B(d|n)=\mathfrak{osp}(2d+1|2n)$ are the roots of $\mathfrak{so}(2d+1)\oplus \mathfrak{sp}(2n)$. The odd roots are given by $\pm (\epsilon_j-\delta_i), \pm \delta_i, \pm (\epsilon_j+\delta_i)$ for $1\le j\le d$ and $1\le i\le n$. The standard choice of positive roots (see \cite{MR1773773, MR051963}) corresponds to the standard choice for $\mathfrak{so}(2d+1)\oplus \mathfrak{sp}(2n)$ and $\delta_i-\epsilon_j,\delta_i,\epsilon_j+\delta_i$ for the odd roots. 

Another useful choice of positive roots corresponds to the one above, except $\epsilon_j-\delta_i$ is chosen instead of $\delta_i-\epsilon_j$, see e.g. \cite{OSpHarm, MR2395482}. The choice of simple positive roots is then given by
\begin{eqnarray*}
\alpha_j=\epsilon_j-\epsilon_{j+1} \quad j=1,\cdots,d-1 &&\alpha_d=\epsilon_d-\delta_1 \\
\alpha_{d+i}=\delta_i-\delta_{i+1}\quad i=1,\cdots,n-1&&\alpha_{d+n}=\delta_n.
\end{eqnarray*}
Remark that here two simple positive roots are odd, $\alpha_d$ and $\alpha_{d+n}$.

The Chevalley basis corresponding to this root system is:
\begin{eqnarray}
\nonumber
X_{\alpha_j}&=&E_{j,j+1}-E_{d+j+1,d+j}\qquad \qquad j=1,\cdots,d-1\\
\label{posrootsbeta}
X_{\alpha_d}&=&E_{d,m+1}-E_{m+n+1,2d}\qquad \qquad Y_{\alpha_d}=-\left(E_{m+1,d}+E_{2d,m+n+1}\right)\\
\nonumber
X_{\alpha_{d+i}}&=&E_{m+i,m+i+1}-E_{m+n+i+1,m+n+i}\qquad i=1,\cdots,n-1\\
\nonumber
X_{\alpha_{d+n}}&=&E_{m,m+2n}+E_{m+n,m}\qquad \qquad Y_{\alpha_{d+n}}=-E_{m+2n,m}+E_{m,m+n}
\end{eqnarray}
and $Y_{\alpha_k}=X_{\alpha_k}^{T}$ for the even roots. The basis for $\mathfrak{h}$ than becomes
\begin{eqnarray*}
H_{\alpha_j}&=&E_{j,j}-E_{j+1,j+1}-E_{j+d,j+d}+E_{j+d+1,j+d+1} \quad \qquad j=1,\cdots,d-1\\
H_{\alpha_d}&=&E_{m+n+1,m+n+1}-E_{m+1,m+1}-E_{d,d}+E_{2d,2d}\\
H_{\alpha_{d+i}}&=&E_{m+i,m+i}-E_{m+i+1,m+i+1}-E_{m+n+i,m+n+i}+E_{m+n+i+1,m+n+i+1}\quad i=1,\cdots,n-1\\
H_{\alpha_{d+n}}&=&E_{m+n,m+n}-E_{m+2n,m+2n}.
\end{eqnarray*}
They satisfy the following relations:
\begin{eqnarray}
\nonumber
[X_{\alpha_k},Y_{\alpha_l}]&=&\delta_{kl}H_{\alpha_k}\qquad\qquad k,l=1,\cdots,d+n\\
\label{Cartanrel}
\,[H,X_{\alpha_k}]&=&\alpha_k(H)X_{\alpha_k}\qquad k=1,\cdots,d+n \qquad\mbox{and}\quad H\in\mathfrak{h}\\
\nonumber
\,[H,Y_{\alpha_k}]&=&-\alpha_k(H)Y_{\alpha_k}\qquad k=1,\cdots,d+n \qquad\mbox{and}\quad H\in\mathfrak{h}.
\end{eqnarray}

For $B(0|n)=\mathfrak{osp}(1|2n)$ with the identification $\mathfrak{o}(1)=\emptyset$, this choice of positive roots and simple roots $(\alpha_1,\cdots,\alpha_n)$ is identical to the standard distinguished basis, so $L^{1|2n}_\Lambda=K^{1|2n}_\Lambda$. 

We also give the positive simple roots for the standard positive root system:
\begin{eqnarray*}
\beta_i=\delta_i-\delta_{i+1} \quad i=1,\cdots,n-1 &&\beta_n=\delta_n-\epsilon_1 \\
\beta_{n+j}=\epsilon_j-\epsilon_{j+1}\quad j=1,\cdots,d-1&&\beta_{n+d}=\epsilon_d.
\end{eqnarray*}
\subsection{The case $\mathfrak{osp}(2d|2n)$}

The even roots of $D(d|n)=\mathfrak{osp}(2d|2n)$ are the roots of $\mathfrak{so}(2d)\oplus \mathfrak{sp}(2n)$. The odd roots are given by $\pm (\epsilon_j-\delta_i), \pm (\epsilon_j+\delta_i)$. The standard choice of positive roots (see \cite{MR1773773, MR051963}) corresponds to the standard choice for $\mathfrak{so}(2d)\oplus \mathfrak{sp}(2n)$ and $\delta_i-\epsilon_j,\epsilon_j+\delta_i$ for the odd roots. 

Another useful choice of positive roots corresponds to the one above except again $\epsilon_j-\delta_i$ is chosen instead of $\delta_i-\epsilon_j$, see e.g. \cite{OSpHarm, MR2395482}. The choice of simple positive roots is then given by
\begin{eqnarray*}
\alpha_j=\epsilon_j-\epsilon_{j+1} \quad j=1,\cdots,d-1 &&\alpha_d=\epsilon_d-\delta_1 \\
\alpha_{d+i}=\delta_i-\delta_{i+1}\quad i=1,\cdots,n-1&&\alpha_{d+n}=2\delta_n.
\end{eqnarray*}
Only the last simple root differs symbolically from the corresponding root for $\mathfrak{osp}(2d+1|2n)$ and only one positive simple root is odd.

The positive and negative root vectors for the case $m=2d$ are given by the corresponding expressions as in the case $m=2d+1$ \eqref{posrootsbeta} with now identification $m=2d$, except for
\begin{eqnarray*}
X_{\alpha_{d+n}}=E_{m+n,m+2n}&&Y_{\alpha_{d+n}}=E_{m+2n,m+n}.
\end{eqnarray*}
Also the corresponding bases for the Cartan subalgebra are of the exact same form as in the case $m=2d+1$. The Chevalley-basis again satisfies relations \eqref{Cartanrel}.

For the case $C(n+1)=\mathfrak{osp}(2|2n)$, the standard distinguished choice of positive roots and simple roots in \cite{MR1773773, MR051963} is equal to the one made here. So there is no need to pay attention to the choice where $\delta_i-\epsilon$ would be considered as a positive root. The notation `$L^{2|2n}_{\Lambda}$' will therefore not be used, representations will always be denoted by $K^{2|2n}_{\Lambda}$.

We also give the positive simple roots for the standard positive root system:
\begin{eqnarray*}
\beta_i=\delta_i-\delta_{i+1} \quad i=1,\cdots,n-1 &&\beta_n=\delta_n-\epsilon_1 \\
\beta_{n+j}=\epsilon_j-\epsilon_{j+1}\quad j=1,\cdots,d-1&&\beta_{n+d}=\epsilon_{d-1}+\epsilon_d.
\end{eqnarray*}

\section{Spinor representations for $\mathfrak{so}(m)$ and $\mathfrak{sp}(2n)$}
\label{classicalspinor}

In this section we recall some facts about the spinor representations for the complex Lie algebras $\mathfrak{so}(m)$ and $\mathfrak{sp}(2n)$. The explicit realizations of these representations will not be repeated since they can be deduced from the spinor representations for $\osp$ in Section \ref{sectionsuperspin}.

The irreducible representation with highest weight $\lambda$ will be denoted by $L^{m|0}_\lambda$ for $\mathfrak{so}(m)$ and by $L_\lambda^{0|2n}$ for $\mathfrak{sp}(2n)$, in correspondence with the notation $L^{m|2n}_\Lambda$ for $\osp$-representations.

The spinor representations for $\mathfrak{so}(m)$ with $m>2$ are realizations of $\mathfrak{so}(m)$ as differential operators on the Grassmann algebra $\Lambda_{d}$ generated by $d=\lfloor m/2\rfloor$ anti-commuting variables. For $m=2d+1$ this Grassmann algebra is an irreducible module with highest weight $\omega_d$ (see Definition \ref{classweights}) and we write
\begin{eqnarray*}
\mS_{2d+1|0}\cong L^{2d+1|0}_{\omega_d}&\cong&\Lambda_{d}.
\end{eqnarray*}
For $\mathfrak{so}(2d)$ the Grassmann algebra splits up into the spaces of elements with even and odd number of generators, then
\begin{eqnarray*}
\mS_{2d|0}=\mS_{2d|0}^+\oplus\mS_{2d|0}^-\cong L_{\omega_d}^{2d|0}\oplus L_{\omega_{d-1}}^{2d|0}\cong\Lambda_d.
\end{eqnarray*}
All of these representations satisfy the fact that all weight spaces have dimension one, see the subsequent equation \eqref{weightmonomials}. They exponentiate to representations of the spin group $Spin(m)$, the double cover of the special orthogonal group $SO(m)$.

An $\mathfrak{so}(2d+1)$-weight $\lambda =\sum_{j=1}^dk_j\epsilon_j$ is integral if either all $k_j$ are integers or half integers. In order to be a dominant weight
\begin{eqnarray*}
k_1\ge k_2\ge\cdots\ge k_{d-1}\ge k_d\ge 0
\end{eqnarray*}
needs to hold.

We can write a weight in terms of the fundamental weights, see Definition \ref{classweights}. For an integral (not necessarily dominant) $\mathfrak{so}(2d+1)$-weight $\lambda=\sum_{j=1}^dk_j\epsilon_j$ we introduce the (not necessarily positive) integers $\lambda_j$, $j=1,\cdots,d$ defined by $\lambda=\sum_{j=1}^d\lambda_j\omega_j$. This implies
\begin{eqnarray*}
\lambda_j=k_j-k_{j-1}\quad\mbox{for}\quad 1\le j\le d-1\quad\mbox{and}\quad\lambda_d=2k_d.
\end{eqnarray*}
For an integral dominant $\mathfrak{so}(2d+1)$-weight $\lambda$, the $\lambda_j$ are non-negative and we define the following set of weights
\begin{eqnarray}
\label{setweightsclass}
I_\lambda&=&\{\sum_{j=1}^di_j\epsilon_j=\sum_{j=1}^d\mu_j\omega_j\,|\,0\le i_j\le 1\quad\mbox{and}\quad\mu_j\le \lambda_j\quad\mbox{for}\quad 1\le j\le d\}.
\end{eqnarray}

The condition on an $\mathfrak{so}(2d)$-weight $\lambda=\sum_{j=1}^dk_j\epsilon_j$ to be dominant is different, the property $k_1\ge k_2\ge \cdots \ge k_{d-1}\ge |k_d|$ must hold. This corresponds to the different fundamental weights in Remark \ref{classweights}. The equality $\sum_{j=1}^dk_j\epsilon_j=\sum_{j=1}^d\lambda_j\omega_j$ then leads to the same expression for $\lambda_j$ as in the case $\mathfrak{so}(2d+1)$ except $\lambda_{d}=k_{d-1}+k_d$. However, the weights which are relevant in this paper will satisfy $k_d\ge 0$ and also the property that $k_{d}=0$ implies $k_{d-1}=0$. For those weights $\lambda$, the set \eqref{setweightsclass} is identical if $\lambda_{d}$ is defined as $2k_d$ or $k_{d-1}+k_d$. We will always assume that integral dominant $\mathfrak{so}(2d)$-weight satisfies these properties.


The decomposition of the tensor product of the spinor representations with finite dimensional modules $L^{m|0}_{\lambda}$, $\lambda=\sum_j k_j\epsilon_j$ with $k_j$ integers, can be calculated using Klimyk's formula, as is done in detail in theorem 1 in \cite{MR2782791} for $L_{k\epsilon_1+l\epsilon_2}^{m|0}$. However, the spinor representations are exactly the miniscule representations for $\mathfrak{so}(m)$. This implies that the decomposition of the tensor product of $L_{\lambda}^{m|0}$ with a spinor representation is given by a multiplicity free decomposition with highest weights given by the sum of $\lambda$ with the weights appearing in the spinor representation, such that the resulting weight is dominant. A written proof of this can be found in  lemma 11 in \cite{MR1401053} for the case $\mathfrak{gl}(l,\mF_p)$, see also proposition 3.7 in \cite{MR1738448}. This approach is used in the following theorem for $\mathfrak{so}(2d+1)$.

\begin{theorem}
\label{classdecomp}
For an integral dominant $\mathfrak{so}(m)$-highest weight $\lambda=\sum_{j=1}^dk_j\epsilon_j$, with $d=\lfloor m/2\rfloor$, $k_d$ an integer (and $k_{d-1}=0$ if $k_{d}=0$ in case $m=2d$), the decomposition
\begin{eqnarray*}
\mS_{m|0}\otimes L^{m|0}_{\lambda}&=&\bigoplus_{\mu \in I_\lambda} L^{m|0}_{\lambda-\mu+\omega_d}
\end{eqnarray*}
holds for $I_\lambda$ in equation \eqref{setweightsclass}.
\end{theorem}
\begin{proof}
We consider $m=2d+1$, which implies $\mS_{m|0}\cong L^{m|0}_{\omega_d}$. As argued above, the weights appearing in the summation of the decomposition are those of the form $\lambda+\nu$ such that 
\begin{itemize}
\item$\nu$ is a weight appearing in the spinor space $L^{2d+1|0}_{\omega_d}$
\item $\lambda+\nu$ is dominant.
\end{itemize}
The weights appearing in $L^{2d+1|0}_{\omega_d}$ are those of the form $\omega_d-\sum_{j=1}^d i_j\epsilon_j$ with $0\le i_j\le 1$. This can for instance be concluded immediately from the subsequent equation \eqref{weightmonomials}. Now $\lambda+\omega_d-\mu$ with $\mu=\sum_{j=1}^di_j\epsilon_j=\sum_{j=1}^d\mu_j\omega_j$ is dominant if and only if
\begin{eqnarray}
\label{domconditie}
\lambda_j+\delta_{j,d}-\mu_j&\ge&0\qquad\mbox{for}\quad 1\le j\le d.
\end{eqnarray}
This corresponds to the conditions in equation \eqref{setweightsclass} when taking into account that the last relation $2k_d+1\ge 2 i_d$ for $k_d$ and $i_d$ integers is equivalent with $\lambda_d=2k_d\ge 2 i_d=\mu_d$.

The proof for $m=2d$ is similar.
\end{proof}

Similarly, the symplectic spinors can be realized as the space of polynomials in $n$ commuting variables, $\mR[t_1,\cdots,t_n]$, see \cite{MR0400304} or \cite{MR1297597}. The symplectic algebra $\mathfrak{sp}(2n)$ is then realized as differential operators in $n$ commuting variables. This corresponds to the Segal-Shale-Weil representation of the metaplectic group $Mp(2n)$, the double cover of the symplectic group $Sp(2n)$. The algebra of polynomials $\mR[t_1,\cdots,t_n]$ decomposes into two irreducible highest weight representations, corresponding to the even and odd polynomials:
\begin{eqnarray*}
\mS_{0|2n}=\mS_{0|2n}^+\oplus\mS_{0|2n}^-\cong L^{0|2n}_{-\frac{1}{2}\nu_{n}}\oplus L^{0|2n}_{\nu_{n-1}-\frac{3}{2}\nu_n}&\cong&\mR[t_1,\cdots,t_n].
\end{eqnarray*}
Contrary to the orthogonal spinor representations, these are infinite dimensional. Theorem 3 in \cite{MR1297597} states that these representations are the only completely pointed (infinite dimensional with all weight spaces having dimension one) highest weight modules for $\mathfrak{sp}(2n)$. For $\mathfrak{so}(m)$ no such completely pointed modules exist.
\begin{theorem}
\label{classcompp}
The only irreducible completely pointed highest weight modules for $\mathfrak{sp}(2n)$ are given by $L^{0|2n}_{-\frac{1}{2}\nu_{n}}$ and $L^{0|2n}_{\nu_{n-1}-\frac{3}{2}\nu_n}$. There are no irreducible completely pointed highest weight modules for $\mathfrak{so}(m)$.
\end{theorem}
Also the decomposition of tensor products of finite dimensional $\mathfrak{sp}(2n)$-representations with the symplectic spinor spaces is calculated in \cite{MR1297597}. 


\section{Irreducible highest weight $\osp$-representations}
\label{sectionrefl}

The irreducible representation $V$ of $\osp$ with unique highest weight $\Lambda\in\mathfrak{h}^\ast$, in the standard choice of positive roots from \cite{MR051963}, will be denoted by $L_\Lambda^{m|2n}$. That same representation is also an irreducible highest weight module with respect to our choice of positive roots given in Section \ref{sectionroot}. The highest weight is different with respect to this choice of positive roots, $\mu$, and we denote the representation also by $K_\mu^{m|2n}$. Calculating $\mu$ from $\Lambda$ and vice versa can be done using the technique of odd reflections from \cite{MR1327543, MR1201236}. For the cases $C(n)=\mathfrak{osp}(2|2n-2)$ and $B(0|n)=\mathfrak{osp}(1|2n)$ this is not necessary, as explained in Section \ref{sectionroot}.

In order to generalize Theorem \ref{classdecomp} we are interested in finite dimensional representations of the form $K^{m|2n}_{\lambda}$ with $\lambda=\sum_{j=1}^dk_j\epsilon_j$ an integral dominant $\mathfrak{so}(m)$-weight. Since we want $K^{m|2n}_{\lambda}$ to be finite dimensional, the $k_j$ need to be integers, not half integers.

We need to know what the highest weight of these representations is in the standard choice of positive roots. The case $K^{m|2n}_{k\epsilon_1}$ (for both $m=2d$ and $m=2d+1$) was already obtained in Remark 8 in \cite{OSpHarm},
\begin{equation}
\label{KLold}
K^{m|2n}_{k\epsilon_1}=\begin{cases}
L^{m|2n}_{\nu_k}&\mbox{if}\quad k\le n\\
L^{m|2n}_{(k-n)\epsilon_1+\nu_n}&\mbox{if}\quad k> n.
\end{cases}
\end{equation}

For the more general case we use the method of odd reflections. We start from one choice of positive roots for which we have an irreducible representation with unique highest weight $\Lambda$. The procedure from \cite{MR1327543, MR1201236} describes that if we replace one positive odd root $\alpha$ by its negative $-\alpha$ the highest weight of the representation, with respect to the new choice of positive roots, becomes $\Lambda -\alpha$ if $\langle\Lambda,\alpha\rangle\not=0$ and stays $\Lambda$ if $\langle\Lambda,\alpha\rangle=0$.

Going from our choice of positive roots to the standard choice corresponds to switching (in this order)
\begin{eqnarray}
\label{orderrefl}
\epsilon_d-\delta_1, \epsilon_d-\delta_2,\cdots,\epsilon_d-\delta_n,\epsilon_{d-1}-\delta_1,\cdots,\epsilon_1-\delta_n.
\end{eqnarray}

For completeness we will give the corresponding highest weights in both root systems for all irreducible finite dimensional highest weight $\osp$-modules, the proof is an direct calculation using the technique of odd reflections. We again assume that $k_d\ge 0$ holds and that $k_{d}=0$ implies $k_{d-1}=0$ for $\mathfrak{so}(2d)$, with notations explained in the theorem.
\begin{theorem}
\label{changeroot}
Each finite dimensional $\osp$-representation $K^{m|2n}_{\mu}$ with
\begin{eqnarray*}
\mu&=&\sum_{j=1}^dk_j\epsilon_j+\sum_{i=1}^nl_i\delta_i\qquad\mbox{(where $l_{k_d+1}=0$ must hold if $k_d<n$ by consistency)},
\end{eqnarray*}
is identical to the highest weight representation $L^{m|2n}_\Lambda$ with
\begin{eqnarray*}
\Lambda&=&\sum_{j=1}^a(k_j-n)\epsilon_j+\sum_{i=1}^nl_{i}\delta_i+\sum_{j=a+1}^d\nu_{k_j}+a\nu_n,
\end{eqnarray*}
where $a$ is defined as the largest integer such that $k_a\ge n$. So in particular $a=0$ means $k_1< n$ and $a=d$ means $k_d \ge n$.
\end{theorem}
The consistency condition `$l_{k_d+1}=0$ must hold if $k_d<n$' on $\mu$ can be derived immediately from applying the procedure of odd reflections. We need to check that the resulting highest weight $\Lambda$ satisfies the consistency conditions on a highest weight in the distinguished root system given in e.g. chapter 36 in \cite{MR1773773} or in \cite{MR051963}. For $\mathfrak{osp}(2d+1|2n)$, a dominant weight $\Lambda=\sum_{j}t_j\epsilon_j+\sum_{i} s_i\delta_i$ must satisfy $t_{s_n+1}=0$ if $s_n<d$. In our case $s_n=a+l_n$, if $s_n<d$ this implies $a<d$ and therefore $l_n=0$, so in that case $s_n=a$ holds and $t_{a+1}$ is clearly zero. So the outcome of the technique of odd reflections exactly gives all the consistent weights. The consistency conditions for $\mathfrak{osp}(2d|2n)$ holds because of the assumption that $k_d\ge 0$ holds and that $k_{d}=0$ implies $k_{d-1}=0$.

For the case of interest in this paper Theorem \ref{changeroot} yields
\begin{eqnarray}
\label{KL}
K_{\sum_{j=1}^dk_j\epsilon_j}^{m|2n}&=&L^{m|2n}_{\sum_{j=1}^a (k_j-n)\epsilon_j+\sum_{j=a+1}^d \nu_{k_j}+a\nu_n},
\end{eqnarray}
as an extension of equation \eqref{KLold}.

\section{Complete reducibility}
\label{criterion}

In this section we consider a basic classical Lie superalgebra $\mathfrak{g}$ and derive a criterion for the complete reducibility of the tensor product of two irreducible (not necessarily finite dimensional) highest weight representations. The root space decomposition of $\mg$ is given by $\mathfrak{n}^++\mathfrak{h}+\mathfrak{n}^-$ and $M_\lambda$ denotes the irreducible highest weight representation with highest weight $\lambda$. The set of positive roots is given by $\Delta^+\subset\mathfrak{h}^\ast$. For each positive root $\alpha\in\Delta^+$ we fix the positive root vector $X_\alpha\in\mathfrak{n}^+$ and the negative root vector $Y_\alpha\in\mathfrak{n}^-$.

Classically each finite dimensional representation of a semisimple complex Lie algebra has a contravariant inner product, induced from the invariant inner product of the compact real form. This implies complete reducibility of the tensor products. In this section we mainly investigate the consequences of the non-degenerate contravariant hermitian form that we construct for irreducible representations of Lie superalgebras.

The approach we take makes a link between primitive vectors of a certain representation $W$ and vectors which can not be obtained from the action of negative root vectors on other vectors in $W$. The representation $W$ is completely reducible if and only if the space of vectors which can be obtained from the action of negative root vectors, denoted by $\mathfrak{n}^-\cdot W$ and the space of primitive vectors denoted by $A(\mathfrak{n}^+)$ satisfy
\begin{eqnarray*}
\mathfrak{n}^-\cdot W\,\oplus \, A(\mathfrak{n}^+)&=&W.
\end{eqnarray*}
So when the representation is completely reducible, the primitive vectors can not be obtained from action of negative root vectors and such primitive vectors are known as maximal vectors. Even though the tensor product is not always completely reducible we will be able to prove that the dimensions of these spaces (when restricted to a weight space $W_\nu$) still satisfy the corresponding property. 

To express the results more elegantly we introduce the notations $\left(\mathfrak{n}^-\cdot W\right)_\nu=(\mathfrak{n}^-\cdot W)\cap W_\nu$ and $A(\mathfrak{n}^+)_\nu=A(\mathfrak{n}^+)\cap W_\nu$ for each weight $\nu$ appearing in $W$.

\begin{theorem}
\label{dimensequal}
Consider the tensor product $W=M_\lambda\otimes M_\mu$ of two irreducible highest weight modules of the Lie superalgebra $\mg$. When restricting to a certain weight space $W_\nu$, the dimension of the space of primitive vectors is equal to the codimension of the space of vectors that can be obtained from the action of negative root vectors on other vectors in $W$. If $W$ is assumed to have bounded multiplicities this can be expressed as
\begin{eqnarray*}
\dim\left(A(\mathfrak{n}^+)_\nu\right)+\dim\left(\left(\mathfrak{n}^-\cdot W\right)_\nu\right)&=&\dim W_\nu.
\end{eqnarray*}
\end{theorem}
\begin{proof}
The representation $W$ has a non-degenerate hermitian form $(\cdot,\cdot)$, in fact non-degenerate on each weight space, such that for each negative root vector $Y_\alpha$ with corresponding positive root vector $X_\alpha$, the relation
\begin{eqnarray*}
(Y_\alpha x,y)&=&(-1)^{|\alpha||x|}(x,X_\alpha y)
\end{eqnarray*}
holds for homogeneous vectors $x,y\in W$. This hermitian form is given by the product of the hermitian forms on $M_\lambda$ and $M_\mu$ in the subsequent Lemma \ref{lemmaBilForm}, $(a\otimes b,c\otimes d)=(-1)^{|b||c|}(a,c)(b,d)$, for $a,c\in M_\lambda$ and $b,d\in M_\mu$.

Using the fact that $\mathfrak{n}^-\cdot W=$Span$\{Y_\alpha x| x\in W,\,\alpha\in \Delta^+\}$ and the non-degeneracy it follows immediately that the vector space
\begin{eqnarray*}
\left(\mathfrak{n}^-\cdot W\right)^\perp&:=& \{y\in W|(y,z)=0,\,\forall z\in \mathfrak{n}^-\cdot W\}
\end{eqnarray*}
is equal to $A(\mathfrak{n}^+)$. Since $(\cdot,\cdot)$ is non-degenerate on each weight space the conclusion on the dimensions follows immediately.
\end{proof}


\begin{remark}
Theorem \ref{dimensequal} is non-trivial since it does not hold for general weight representations, even for ordinary Lie algebras. For example, in a Verma module $V$ the codimension of $\mathfrak{n}^-\cdot V$ is one while the dimension of the space of primitive vectors can be higher.
\end{remark}

\begin{corollary}
\label{ThmComRed}
Consider the tensor product $W=M_\lambda\otimes M_\mu$ of two irreducible highest weight representations of the Lie superalgebra $\mg$. If the dimension of the space of primitive vectors of $W$ is a finite number $p$ and there is a basis of primitive vectors $\{v_j^+, j=1,\cdots,p\}$ (which are of weight $\lambda_j$) such that
\begin{eqnarray*}
 v_j^+&\not\in&\cU(\mg)\cdot v_k^+ \qquad\mbox{for } j,k=1,\cdots,p\quad\mbox{with}\quad j\not=k,
\end{eqnarray*} 
then $W$ is completely reducible and $W\cong\bigoplus_{j=1}^pM_{\lambda_j}$.
\end{corollary}
\begin{proof}
Each representation $\cU(\mg)\cdot v_k^+ $ is irreducible since it is the quotient of a Verma module with only one primitive vector. Therefore $\cU(\mg)\cdot v_j^+\cong M_{\lambda_j}$. It is then also clear that $\cU(\mg)\cdot v_j^+\cap \cU(\mg)\cdot v_k^+=0$ if $j\not=k$.
It remains to be proved that the representation
\begin{eqnarray*}
V&=&\bigoplus_{j=1}^n\cU(\mg)\cdot v_j^+\cong \bigoplus_{j=1}^nM_{\lambda_j}
\end{eqnarray*}
corresponds to the entire representation $W$. 

First we prove that the $p$ primitive vectors are all vectors which can not be obtained from the action of negative root vectors on other vectors in $W$, i.e. $v_j^+\not\in \mathfrak{n}^-\cdot W$, or every primitive vector is a maximal one. We choose the ordering of the $v_j^+$ in a way that $\lambda+\mu=\lambda_1\ge \lambda_2\ge\cdots\ge\lambda_p$ holds. Obviously the vector $v_1^+$, which corresponds to the product of the maximal vectors of $M_\lambda$ and $M_\mu$, is not in $\mathfrak{n}^-\cdot W$. If the vector $v_2^+$ (or $v_j^+$ with $\lambda_j=\lambda_2$) is generated by $\mathfrak{n}^-$-action on vectors with higher weights, there need to be vectors in $W$ of weight higher than $\lambda_2$, that are not in $\cU(\mathfrak{n}^-)\cdot v_1^+$, since we already know that $v_2^+\not\in\cU(\mathfrak{n}^-)\cdot v_1^+$. In the set of vectors in $W\backslash V$ of weight higher than $\lambda_2$, we take the one with highest weight. This vector is not in $\mathfrak{n}^-\cdot W$, but this is a contradiction with Theorem \ref{dimensequal} because there is only one primitive vector with weight higher than $\lambda_2$ and there is already a vector ($v_1^+$) of weight higher than $\lambda_2$ which is not in $\mathfrak{n}^-\cdot W$. Therefore $v_2^+$ is not in $\mathfrak{n}^-\cdot W$. Continuing this procedure until $v_p^+$ shows that all $v_j^+\not\in$ $\mathfrak{n}^-\cdot W$.

Now, if $W\backslash V\not=\emptyset$, we can take a highest weight vector in this set, which we denote by $x$. This vector can not be generated by action of $\mathfrak{n^-}$ on higher weight vectors, since all higher weight vectors are inside the representation $V$. This implies that the dimension of the space of vectors in $W$ that can not be obtained from action of $\mathfrak{n}^-$ is at least $p+1$, while the dimension of the space of primitive vectors is only $p$. This contradicts Theorem \ref{dimensequal}, so $x$ does not exist, which yields $V=W$.
\end{proof}

\begin{remark}
Corollary \ref{ThmComRed} does not hold for general weight representations. Contrary to Theorem \ref{dimensequal}, it still holds if $W$ is replaced by the quotient of a Verma module of a Lie (super)algebra. An easy example of a weight representation $W$ that does not satisfy Corollary \ref{ThmComRed} is given by taking the quotient of the tensor product in Theorem \ref{tensornotcr} with respect to the irreducible subrepresentation.
\end{remark}

Now we start to construct the contravariant hermitian form needed in the proof of Theorem \ref{dimensequal}, therefore we use the following definition.
\begin{definition}
\label{n+n-}
The anti-involution $\tau:\mg\to\mg$ is defined by $\tau(X_\alpha)=Y_\alpha$ and extended to the universal enveloping algebra. In particular for $f(\mathfrak{n}^-)\in\cU(\mathfrak{n}^-)$, of the form $f(\mathfrak{n}^-)=Y_\alpha g(\mathfrak{n}^-)$, $\tau$ satisfies
\begin{eqnarray*}
\tau\left(Y_\alpha g(\mathfrak{n}^-)\right)&=&(-1)^{|\alpha||g|}\tau\left(g(\mathfrak{n}^-)\right)X_\alpha.
\end{eqnarray*}
\end{definition}

The contravariant hermitian form can then be obtained from a Harish-Chandra morphism \[\cU(\mg)/(\mn^-\cU(\mg)+\cU(\mg)\mn^+)\to\cU(\mathfrak{h}).\] This is done explicitly in the following lemma.

\begin{lemma}
\label{lemmaBilForm}
An irreducible highest weight module $M_\lambda$ of a Lie superalgebra $\mg$ has a non-degenerate hermitian form $(\cdot,\cdot)$, such that for each negative root vector $Y_\alpha$ with corresponding positive root vector $X_\alpha$, the relation
\begin{eqnarray*}
(Y_\alpha x,y)&=&(-1)^{|\alpha||x|}(x,X_\alpha y)
\end{eqnarray*}
holds for homogeneous vectors $x,y\in M_\lambda$. The form is also non-degenerate when restricted to each weight space of $M_\lambda$.
\end{lemma}
\begin{proof}

We define the form $(\cdot,\cdot)$ on $M_\lambda$ as follows. The highest weight vector $v^+$ of $M_\lambda$ satisfies $(v^+,v^+)=1$ and $v^+$ is orthogonal with respect to all vectors of lower weight. For all $f,g\in\cU(\mathfrak{n}^-)$ 
\begin{eqnarray*}
(fv^+,gv^+)&=&(v^+,\tau(f)gv^+)
\end{eqnarray*}
holds. This hermitian form satisfies the required properties if it is non-degenerate.

If the hermitian form would be degenerate we denote the vector space of all degenerate vectors by $D$,
\begin{eqnarray*}
D&=&\{x\in M_\lambda|(x,v)=0,\,\forall v\in M_\lambda\}.
\end{eqnarray*}
From the properties of the hermitian form it follows that $D$ is a $\mg$-subrepresentation of $M_\lambda$, but since $M_\lambda$ is irreducible and $(\cdot,\cdot)$ is not identically zero we obtain $D=0$.
\end{proof}

If the hermitian form on $M_\lambda$ in Lemma \ref{lemmaBilForm} is also positive definite (an inner product), the representation is unitary. When two such representations would be considered, Theorem \ref{dimensequal} can be made stronger to $A(\mathfrak{n}^+)_\nu\oplus\left(\mathfrak{n}^-\cdot W\right)_\nu= W_\nu$ and complete reducibility follows immediately, this is the case for Lie algebras and for $\mathfrak{gl}(p|q)$.

When applying Corollary \ref{ThmComRed}, the quadratic Casimir operator $\cC_2\in\cU(\mathfrak{g})$ can be of importance. This quadratic operator is of the form
\begin{eqnarray*}
\cC_2&=&\sum_\alpha Y_\alpha X_\alpha +p(\mathfrak{h})
\end{eqnarray*}
for some quadratic $p(\mathfrak{h})\in\cU(\mathfrak{h})$. The Cartan algebra part satisfies $p(\mathfrak{h})v_\kappa=\langle\kappa,\kappa+2\rho\rangle v_\kappa$ for $v_\kappa$ a vector of weight $\kappa$ in some representation, with $\langle\cdot,\cdot\rangle$ given in equation \eqref{weightinprod} and $\rho$ given by
\[\rho=\sum_{j=1}^d(\frac{m}{2}-j)\epsilon_j+\sum_{i=1}^n(1+n-\frac{m}{2}-i)\delta_i,
\]
for $\mg=\osp$. Since $\cC_2$ commutes with $\mathfrak{g}$ every vector inside an irreducible highest weight representation $M_\lambda$ is an eigenvector of this Casimir operator with the same eigenvalue, so $\cC_2 M_\lambda = \langle \lambda,\lambda+2\rho\rangle M_\lambda$. A necessary condition for $v_k^+\not\in\cU(\mg)\cdot v_j^+$ to hold with notations from Corollary \ref{ThmComRed} is therefore $\langle\lambda_j,\lambda_j+2\rho\rangle\not=\langle\lambda_k,\lambda_k+2\rho\rangle$.

The following theorem shows what happens if the condition of Corollary \ref{ThmComRed} is not met in the simplest case.
\begin{theorem}
\label{Simplecasencr}
Consider the tensor product $W=M_\lambda\otimes M_\mu$ of two irreducible highest weight modules of a Lie superalgebra $\mg$. If the dimension of the space of primitive vectors of $W$ is a finite number $p$ and there is a basis of primitive vectors $\{v_j^+, j=1,\cdots,p\}$ (of strictly different weights $\lambda_j$) such that
\begin{eqnarray*}
 v_j^+&\not\in&\cU(\mg)\cdot v_k^+ \qquad\mbox{for } j,k=1,\cdots,p\quad\mbox{with}\quad j\not=k \quad\mbox{ except when }j=p \mbox{ and } k=p-1
\end{eqnarray*} 
and $v_p^+\in\cU(\mathfrak{\mg})\cdot v_{p-1}^+$, then the decomposition
\begin{eqnarray*}
W&\cong&\left(\bigoplus_{j=1}^{p-2}M_{\lambda_j}\right)\,\bigoplus\, P
\end{eqnarray*}
holds with $P$ having subrepresentations $P\supset V\supset M_{\lambda_{p}}$. The representations $P$ and $V$ are indecomposable and $V$ satisfies $V/M_{\lambda_p}\cong M_{\lambda_{p-1}}$. The representation $P/V$ is a quotient of the Verma module with highest weight $\lambda_p$.
\end{theorem}
\begin{proof}
Each representation $\cU(\mathfrak{n}^-)\cdot v_j^+$ for $j<p-1$ is irreducible since it contains no other primitive vectors. The restriction of $( \cdot,\cdot)$ from the proof of Theorem \ref{dimensequal} is still non-degenerate when restricted to $M_{\lambda_j}= \cU(\mathfrak{n}^-)\cdot v_j^+$ for $j<p-1$ since the subspace of degenerate vectors would constitute a subrepresentation. So the restriction either has to be non-degenerate or zero. It can not be zero since there has to be a vector $a\in W$ of weight $\lambda_j$ such that $( v_j^+,a)\not=0$. This vector can not be in $\mathfrak{n}^-\cdot W$, so by Theorem \ref{dimensequal} it has to contain a part $v_j^+$ and $( v_j^+,v_j^+)\not=0$. Therefore the orthogonal complement of $\left(\bigoplus_{j=1}^{p-2}M_{\lambda_j}\right)$ is denoted by $P$ and $P\cap \left(\bigoplus_{j=1}^{p-2}M_{\lambda_j}\right)=0$, so it satisfies $W\cong\left(\bigoplus_{j=1}^{p-2}M_{\lambda_j}\right)\bigoplus P$.

Now we look at the representation $P$, it contains two primitive vectors $v_{p-1}^+$ and $v_p^+$ such that $v_p^+\in\cU(\mathfrak{n}^-)\cdot v_{p-1}^+$. Theorem \ref{dimensequal} implies that there is a vector of weight $\lambda_p$ which is not generated by $\mathfrak{n}^-$-action on other vectors. If $P$ could be decomposed into two subrepresentations, each representation would have a maximal vector, while there is only one maximal vector inside $P$.

Because $V$ is the quotient of a Verma module with two primitive vectors, it follows immediately that $V$ is indecomposable and $ \cU(\mathfrak{n}^-)\cdot v_p^+\cong M_{\lambda_p}$ and $V/M_{\lambda_p}\cong M_{\lambda_{p-1}}$. The highest weight vector in $P/V$ is of weight $\lambda_p$ and this weight space has dimension 1 in $P/V$ according to Theorem \ref{dimensequal}. If there would be a vector in $P/V$ that is not generated by the highest weight vector, this would lead to another vector in $P$, which is not generated by $\mathfrak{n}^-$-action, which is impossible. This proves that $P/V$ is the quotient of a Verma module.
\end{proof}



\section{Spinor representations for $\osp$}
\label{sectionsuperspin}

Before we introduce spinor representations for $\osp$ we characterize the completely pointed modules for $\osp$. Since the spinors for $\osp$ are a generalization of those for $\mathfrak{sp}(2n)$ they should also constitute completely pointed highest weight modules, see Theorem \ref{classcompp}.

\begin{theorem}
\label{supercompp}
The only irreducible completely pointed highest weight module for $B(0|n)=\mathfrak{osp}(1|2n)$ is given by $L_{-\frac{1}{2}\nu_n}^{1|2n}$.

The only irreducible completely pointed highest weight module for $B(d|n)=\mathfrak{osp}(2d+1|2n)$ is given by $L_{\omega_d-\frac{1}{2}\nu_n}^{2d+1|2n}=K_{\omega_d-\frac{1}{2}\nu_n}^{2d+1|2n}$. 

The only irreducible completely pointed highest weight modules for $C(n+1)=\mathfrak{osp}(2|2n)$ are given by $K_{\frac{1}{2}\epsilon-\frac{1}{2}\nu_n}^{2|2n}$ and $K_{\frac{1}{2}\epsilon+\nu_{n-1}-\frac{3}{2}\nu_n}^{2|2n}$.

The only irreducible completely pointed highest weight modules for $D(d|n)=\mathfrak{osp}(2d|2n)$ are given by $L_{\omega_d-\frac{1}{2}\nu_n}^{2d|2n}=K_{\omega_d-\frac{1}{2}\nu_n}^{2d|2n}$ and $L_{\omega_{d-1}-\frac{1}{2}\nu_n}^{2d|2n}=K_{\omega_d+\nu_{n-1}-\frac{3}{2}\nu_n}^{2d|2n}$.
\end{theorem}
\begin{proof}
We will write the proof in a way that assumes $m >2$, although with simple adjustments of notation it also holds for $B(0|n)$ and $C(n+1)$.

An irreducible completely pointed highest weight $\osp$-module $V$ should decompose as an $\mathfrak{so}(m)\oplus\mathfrak{sp}(2n)$-module into (a finite amount of) irreducible representations which have the property that their weight spaces are one dimensional. Theorem \ref{classcompp} implies that in order to make the representation infinite dimensional at least one of the two completely pointed modules of $\mathfrak{sp}(2n)$ should appear. Therefore we obtain the following decomposition of $V$ as an $\mathfrak{so}(m)\oplus\mathfrak{sp}(2n)$-module,
\begin{eqnarray*}
V&=&U_1\times L^{0|2n}_{-\frac{1}{2}\nu_{n}}\oplus U_2\times L^{0|2n}_{\nu_{n-1}-\frac{3}{2}\nu_n}\oplus U_3\times W.
\end{eqnarray*}
Here $U_1, U_2$ and $U_3$ are finite dimensional $\mathfrak{so}(m)$-representations because of Theorem \ref{classcompp} and $W$ a finite dimensional $\mathfrak{sp}(2n)$-representation. Since $V$ is irreducible we cannot combine integer values for the weights of $\mathfrak{sp}(2n)$ with the half-integer ones for $L^{0|2n}_{-\frac{1}{2}\nu_{n}}$ or $L^{0|2n}_{\nu_{n-1}-\frac{3}{2}\nu_n}$, therefore $W=0$.  So $V$ will correspond to either $K^{m|2n}_{\lambda-\frac{1}{2}\nu_n}$ or $K^{m|2n}_{\lambda+\nu_{n-1}-\frac{3}{2}\nu_n}$, with $\lambda=\sum_{j=1}^dk_j\epsilon_j$ an integral dominant weight for $\mathfrak{so}(m)$.

First we assume that $n>1$. In both cases ($V=K^{m|2n}_{\lambda-\frac{1}{2}\nu_n}$ or $V=K^{m|2n}_{\lambda+\nu_{n-1}-\frac{3}{2}\nu_n}$) there can appear no $\mathfrak{sp}(2n)$-weights which are higher than $-\frac{1}{2}\nu_n$, because of Theorem \ref{classcompp}. Therefore $Y_{\epsilon_j-\delta_1}u$ should be zero for $u$ the highest weight vector of $V$ (since $\delta_1-\frac{1}{2}\nu_n>-\frac{1}{2}\nu_n$ for the first case and since $\delta_1+\nu_{n-1}-\frac{3}{2}\nu_n>-\frac{1}{2}\nu_n$ if $n>1$ for the second case). Since $X_{\epsilon_j-\delta_1}Y_{\epsilon_j-\delta_1}u=-(k_j-\frac{1}{2})u$ for both cases, this immediately implies that $k_j=\frac{1}{2}$, or $\lambda=\omega_d$. This gives the two possibilities for $\mathfrak{osp}(2d|2n)$, we still need to prove that for $\mathfrak{osp}(2d+1|2n)$ only one can appear. 

Assume that $K^{2d+1|2n}_{\omega_d+\nu_{n-1}-\frac{3}{2}\nu_{n}}$ exists and is completely pointed. Since $X_{\epsilon_j-\delta_n}Y_{\epsilon_j-\delta_n}u=-(\frac{1}{2}-\frac{3}{2})u$, we find that $Y_{\epsilon_j-\delta_n}u\not=0$ for $u$ the highest weight vector of $K^{2d+1|2n}_{\omega_d+\nu_{n-1}-\frac{3}{2}\nu_{n}}$. Since $Y_{\epsilon_j-\delta_n}u$ has weight $\omega_d-\epsilon_j-\frac{1}{2}\nu_n$, the $\mathfrak{sp}(2n)$-weight $-\frac{1}{2}\nu_n$ appears, so $L_{-\frac{1}{2}\nu_n}^{0|2n}$ must appear in the $\mathfrak{so}(2d+1)\oplus\mathfrak{sp}(2n)$-decomposition of $K^{2d+1|2n}_{\omega_d+\nu_{n-1}-\frac{3}{2}\nu_{n}}$. As a consequence, the $\mathfrak{so}(2d+1)\oplus \mathfrak{sp}(2n)$ decomposition must be of the form
\begin{eqnarray*}
K^{2d+1|2n}_{\omega_d+\nu_{n-1}-\frac{3}{2}\nu_{n}}&=&L^{2d+1|0}_{\omega_d}\times L^{0|2n}_{\nu_{n-1}-\frac{3}{2}\nu_{n}}\oplus L^{2d+1|0}_{\mu}\times L^{0|2n}_{-\frac{1}{2}\nu_{n}}\oplus\cdots.
\end{eqnarray*}
Here, $\mu$ must be an integral dominant $\mathfrak{so}(2d+1)$-weight, which is strictly lower than $\omega_d$, otherwise the highest weight of the $\mathfrak{osp}(2d+1|2n)$-representation would be $\mu-\frac{1}{2}\nu_{n}$. This is impossible since $\omega_d$ is the lowest integral dominant weight for $\mathfrak{so}(2d+1)$. Therefore only the option $K_{\omega_d-\frac{1}{2}\nu_n}^{2d+1|2n}$ remains. 

Now we consider the case $n=1$. The proof that $V=K^{m|2}_{\lambda-\frac{1}{2}\delta}$ leads to $\lambda=\omega_d$ does not change compared to $n>1$. In case $V=K^{m|2}_{\lambda-\frac{3}{2}\delta}$, the condition that no $\mathfrak{sp}(2n)$-weight higher than $-\frac{1}{2}\delta$ appears leads to the condition $Y_{\epsilon_j-\delta}Y_{\epsilon_l-\delta}u=0$ for $1\le j < l \le d$, or $(k_j-\frac{1}{2})(k_l-\frac{3}{2})=0$. The only possible integral dominant weights that satisfy this are
\[\lambda=\omega_d, \quad \lambda=\omega_{d-1} \,\mbox{ (in case }\, m=2d),\qquad t\epsilon_1+3\omega_d \,\mbox{ for}\, t\in\mN.\]
The second one can be excluded because the technique of odd reflections shows that it leads to an inconsistent highest weight. The third one is not possible since it can be immediately checked it cannot be completely pointed.

Now we assume that $K_{\omega_d-\frac{1}{2}\nu_n}^{2d+1|2n}$, $K_{\omega_d-\frac{1}{2}\nu_n}^{2d|2n}$ and $K_{\omega_d+\nu_{n-1}-\frac{3}{2}\nu_n}^{2d|2n}$ exist. We calculate the highest weight of the representations in the standard choice of positive roots, again using the method from \cite{MR1327543, MR1201236} explained in Section \ref{sectionrefl}. Since $\langle \omega_d-\frac{1}{2}\nu_n,\epsilon_j-\delta_i\rangle=0$ for all $1\le j\le d$ and $1\le i \le n$ we obtain immediately $L_{\omega_d-\frac{1}{2}\nu_n}^{m|2n}=K_{\omega_d-\frac{1}{2}\nu_n}^{m|2n}$. For the case $K_{\omega_d+\nu_{n-1}-\frac{3}{2}\nu_n}^{2d|2n}$ we find 
\begin{eqnarray*}
\langle \omega_d+\nu_{n-1}-\frac{3}{2}\nu_n,\epsilon_d-\delta_i\rangle&=&\frac{1}{2}\left(\frac{1}{2}-\frac{1}{2}-\delta_{i,n}\right)=-\delta_{i,n}\frac{1}{2}.
\end{eqnarray*}
So after applying the first $n$ odd reflections in equation \eqref{orderrefl} we obtain highest weight $\omega_d+\nu_{n-1}-\frac{3}{2}\nu_n-\epsilon_d+\delta_n=\omega_{d-1}-\frac{1}{2}\nu_n$. Since $\langle\omega_{d-1}-\frac{1}{2}\nu_n,\epsilon_j-\delta_i\rangle=0$ for $1\le j\le d-1$ and $1\le i\le n$ the remainder of the odd reflections does not change the highest weight any further.

Finally, the proof that these representations exist and are in fact completely pointed will be clear from the subsequent explicit constructions.
\end{proof}

To realize the completely pointed modules described in Theorem \ref{supercompp} we need to combine the Grassmann algebra and polynomial algebra from Section \ref{classicalspinor} into one superalgebra.
\begin{definition}
\label{superGrass}
The algebra $\Lambda_{d|n}$ is freely generated by $\{\theta_1,\cdots,\theta_d,t_1,\cdots, t_n\}$ subject to the relations
\[\theta_{j}\theta_{k}=-\theta_k\theta_j\quad\mbox{for}\quad  1\le j,k\le d,\qquad t_it_l=t_lt_i\quad\mbox{for}\quad  1\le i,l\le n\]
and
\[\theta_jt_i=-t_i\theta_j\quad\mbox{for}\quad  1\le j\le d,\quad 1\le i\le n.\]
\end{definition}

This algebra is a superalgebra with unusual gradation. The commuting variables are considered as odd and the Grassmann variables are even. With this gradation the algebra is in fact a super anti-commutative algebra, $ab=-(-1)^{|a||b|}ba$ for $a,b$ two homogeneous elements of the superalgebra. Therefore this corresponds to a supersymmetric version of a Grassmann algebra.
\begin{definition}
\label{Xspin1}
The realization $\phi$ of $\mathfrak{osp}(2d+1|2n)$ as endomorphisms on $\Lambda_{d|n}$ is defined by
\begin{eqnarray*}
\phi(X_{\alpha_j})=\theta_{d-j}\partial_{\theta_{d-j+1}}& &\phi(Y_{\alpha_j})=\theta_{d-j+1}\partial_{\theta_{d-j}}\quad j=1,\cdots,d-1\\
\phi(X_{\alpha_d})=t_n\partial_{\theta_{1}}& &\phi(Y_{\alpha_d})=\theta_{1}\partial_{t_n}\\
\phi(X_{\alpha_{d+i}})=t_{n-i}\partial_{t_{n-i+1}}& &\phi(Y_{\alpha_{d+i}})=t_{n-i+1}\partial_{t_{n-i}}\quad i=1,\cdots,n-1\\
\phi(X_{\alpha_{d+n}})=\frac{i}{\sqrt{2}}\partial_{t_1}& &\phi(Y_{\alpha_{d+n}})=\frac{i}{\sqrt{2}}{t_1}.\\
\end{eqnarray*}
\end{definition}
This realization therefore satisfies
\begin{eqnarray*}
\phi(H_{\alpha_j})&=&\theta_{d-j}\partial_{\theta_{d-j}}-\theta_{d-j+1}\partial_{\theta_{d-j+1}}\quad j=1,\cdots,d-1\\
\phi(H_{\alpha_d})&=&t_n\partial_{t_n}+\theta_1\partial_{\theta_{1}}\\
\phi(H_{\alpha_{d+i}})&=&t_{n-i}\partial_{t_{n-i}}-t_{n-i+1}\partial_{t_{n-i+1}}\quad i=1,\cdots,n-1\\
\phi(H_{\alpha_{d+n}})&=&-(t_1\partial_{t_1}+\frac{1}{2}).\\
\end{eqnarray*}

For this realization, the representation of $\mathfrak{osp}(2d+1|2n)$ on $\Lambda_{d|n}$ is a simple highest weight module with highest weight vector $1$. This vector satisfies $\phi(H_{\alpha_k})1=-\frac{1}{2}\delta_{k,d+n}$. Since 
\begin{eqnarray*}
\omega_d(H_{\alpha_k})=-\frac{1}{2}\delta_{k,d}&\mbox{and} &\nu_n(H_{\alpha_k})=-\delta_{k,d}+\delta_{k,d+n},
\end{eqnarray*}
the vector $1$ has weight $\omega_d-\frac{1}{2}\nu_n$, with $\omega_d$ and $\nu_n$ given in Definition \ref{classweights}. We use the notation
\begin{eqnarray*}
\mS_{2d+1|2n}=K^{2d+1|2n}_{\omega_d-\frac{1}{2}\nu_n}=L^{2d+1|2n}_{\omega_d-\frac{1}{2}\nu_n}\cong \Lambda_{d|n},
\end{eqnarray*}
for the spinor representation of $\mathfrak{osp}(2d+1|2n)$.

The weight of elements of $\Lambda_{d|n}$ in the $\mathfrak{osp}(2d+1|2n)$-representation can be calculated from the expressions $\phi(H_\alpha)$. The weight of the vector
\begin{eqnarray}
\label{weightmonomials}
\theta_1^{\gamma_1}\theta_2^{\gamma_2}\cdots\theta_d^{\gamma_d}t_1^{\beta_1}t_2^{\beta_2}\cdots t_n^{\beta_n}&\mbox{is given by}&\omega_d-\frac{1}{2}\nu_n-\sum_{j=1}^d\gamma_{d-j+1}\epsilon_j-\sum_{i=1}^n\beta_{n-i+1}\delta_i.
\end{eqnarray}

Since this constitutes a (Poincar\'e-Birkhoff-Witt type) basis for $\Lambda_{d|n}$, this representation is completely pointed, which completes the proof of Theorem \ref{supercompp} for $\mathfrak{osp}(2d+1|2n)$. Thus we have obtained the unique completely pointed representation for $\mathfrak{osp}(2d+1|2n)$ as the spinor module.

Very similarly we can define a representation of $\mathfrak{osp}(2d|2n)$ on $\Lambda_{d|n}$.
\begin{definition}
\label{Xspin2}
The realization $\varphi$ of $\mathfrak{osp}(2d|2n)$ as endomorphisms on $\Lambda_{d|n}$ is defined by
\begin{eqnarray*}
\varphi(X_{\alpha_j})=\theta_{d-j}\partial_{\theta_{d-j+1}}& &\varphi(Y_{\alpha_{j}})=\theta_{d-j+1}\partial_{\theta_{d-j}}\quad j=1,\cdots,d-1\\
\varphi(X_{\alpha_d})=t_n\partial_{\theta_1}&&\varphi(Y_{\alpha_d})=\theta_1\partial_{t_n}\\
\varphi(X_{\alpha_{d+i}})=t_{n-i}\partial_{t_{n-i+1}}& &\varphi(Y_{\alpha_{d+i}})=t_{n-i+1}\partial_{t_{n-i}}\quad i=1,\cdots,n-1\\
\varphi(X_{\alpha_{d+n}})=-\frac{1}{2}\partial_{t_1}^2&&\varphi(Y_{\alpha_{d+n}})=\frac{1}{2}t_1^2
\end{eqnarray*}
\end{definition}
This definition implies $\varphi(H_{\alpha_k})=\phi(H_{\alpha_k})$ with $\phi(H_{\alpha_k})$ the realization for the corresponding element of $\mathfrak{osp}(2d+1|2n)$.

With this realization, the representation of $\mathfrak{osp}(2d|2n)$ on $\Lambda_{d|n}$ decomposes into two irreducible modules. One consisting of the polynomials of even degree, generated by the highest weight vector $1$ and one of the polynomials of odd degree, generated by highest weight vector $t_1$ (or $\theta_1$ in the standard choice of positive odd roots). We find $\varphi(H_{\alpha_k})1=-\frac{1}{2}\delta_{k,d+n}$ and $\varphi(H_{\alpha_k})t_1=-\frac{3}{2}\delta_{k,d+n}+\delta_{k,d+n-1}$. These correspond to the weights $\omega_d-\frac{1}{2}\nu_n$ and $\omega_{d}+\nu_{n-1}-\frac{3}{2}\nu_n$. Hence we obtain the two representations of $\mathfrak{osp}(2d|2n)$ in Theorem \ref{supercompp}, which proves that they exist. We use the notations
\begin{eqnarray*}
\mS^+_{2d|2n}=L^{2d|2n}_{\omega_d-\frac{1}{2}\nu_n}=K^{2d|2n}_{\omega_d-\frac{1}{2}\nu_n}\cong \Lambda_{d|n}^{+} &\mbox{and}&\mS^-_{2d|2n}=L^{2d|2n}_{\omega_{d-1}-\frac{1}{2}\nu_n}=K^{2d|2n}_{\omega_d+\nu_{n-1}-\frac{3}{2}\nu_n}\cong \Lambda_{d|n}^{-}.
\end{eqnarray*}
The notations $\Lambda_{d|n}^{+}$ and $\Lambda_{d|n}^{-}$ are used for the subalgebras of $\Lambda_{d|n}$ which are generated by an even or odd amount of generators $\theta_j$ or $t_i$ (not to be confused with the even and odd part of $\Lambda_{d|n}$ according to the $\mZ_2$ gradation). We also use the notation 
\[\mS_{2d|2n}=\mS^+_{2d|2n}\oplus\mS^-_{2d|2n}\cong \Lambda_{d|n}.\]
The weight of an element of $\Lambda_{d|n}$ as an $\mathfrak{osp}(2d|2n)$-representation is again given by equation \eqref{weightmonomials}, this is a direct consequence of the relations $\varphi(H_{\alpha_k})=\phi(H_{\alpha_k})$. So also these spinor representations are completely pointed, which concludes the proof of Theorem \ref{supercompp}.

In general we call the space $\mS_{m|2n}$ the super spinor space. This is $\Lambda_{\lfloor m/2\rfloor|n}$ as an $\mathfrak{osp}(m|2n)$-representation, which is irreducible depending on whether $m$ is even or odd. 

Comparing to Section \ref{classicalspinor} shows how $\mS_{m|2n}$ decomposes as an $\mathfrak{so}(m)\oplus \mathfrak{sp}(2n)$-representation.
\begin{eqnarray}
\nonumber
\mS_{2d+1|2n}&=&\mS_{2d+1|0}\times \mS^+_{0|2n}\oplus \mS_{2d+1|0}\times \mS^-_{0|2n}\\
\label{decompSp}
\mS^+_{2d|2n}&=&\mS^+_{2d|0}\times \mS^+_{0|2n}\oplus \mS^-_{2d|0}\times \mS^-_{0|2n}\\
\nonumber
\mS^-_{2d|2n}&=&\mS^+_{2d|0}\times \mS^-_{0|2n}\oplus \mS^-_{2d|0}\times \mS^+_{0|2n}.
\end{eqnarray}
Here we identify $\mS^+_{2|0}=\{1\}$ and $\mS^-_{2|0}=\{\theta_1\}$ for $\mathfrak{osp}(2|2n)$ and for $\mathfrak{osp}(1|2n)$ the decomposition should be $\mS_{1|2n}= \mS^+_{0|2n}\oplus  \mS^-_{0|2n}$. This shows that the decomposition of the $\osp$-representation in $\mathfrak{so}(m)\oplus\mathfrak{sp}(2n)$-representations is very small in a sense. This is natural, since it corresponds to a generalized notion of the minimal representation of $\mathfrak{sp}(2n)$, see \cite{Joseph}.

The representations $\mS_{2d+1|2n}$, $\mS^{\pm}_{2d|2n}$ are unitarizable. In case $m=2d$ this can be seen from proposition 4.1 in \cite{MR2028498}. In general it follows immediately from the inner product on $\Lambda_{d|n}$ defined by
\begin{eqnarray*}
\langle \theta_1^{a_1}\cdots \theta_d^{a_d}t_1^{b_1}\cdots t_n^{b_n}| \theta_1^{s_1}\cdots \theta_d^{s_d}t_1^{r_1}\cdots t_n^{r_n}\rangle&=&b!\,\delta_{as}\delta_{br}.
\end{eqnarray*}

\begin{remark}
Instead of definition  \ref{superGrass} we could also have considered the algebra $\widetilde{\Lambda}_{d|n}$ generated by $\widetilde{\theta}_j$ and $\widetilde{t}_i$ with commutation relations 
\begin{eqnarray*}
\widetilde{\theta_j}\widetilde{\theta}_k=-\widetilde{\theta}_k\widetilde{\theta}_j\qquad\widetilde{t}_i \widetilde{t}_l=\widetilde{t}_l\widetilde{t}_i\quad\mbox{and}\qquad \widetilde{\theta_j}\widetilde{t}_i=\widetilde{t}_i\widetilde{\theta}_j.
\end{eqnarray*} 
The realizations of $\osp$ in definitions \ref{Xspin1} and \ref{Xspin2} can be defined in the exact same way on this algebra, this corresponds to a special case of the oscillator realization in \cite{MR1037401}. The realization of $\osp$ in definitions \ref{Xspin1} and \ref{Xspin2} could also be defined by substituting $(-1)^{\sum_{k=1}^d\widetilde{\theta}_k\partial_{\widetilde{\theta}_k}}\widetilde{t}_i$ for $t_i$ and $(-1)^{\sum_{k=1}^d\widetilde{\theta}_k\partial_{\widetilde{\theta}_k}}\widetilde{\theta}_j$ for $\theta_j$. These operators on $\widetilde{\Lambda}_{d|n}$ generate an algebra isomorphic to $\Lambda_{d|n}$. This corresponds to the oscillator realization in chapter 29 of \cite{MR1773773} and the Howe duality $\osp \supset \mathfrak{sp}(2)\times \mathfrak{osp}(n|2\lfloor m/4\rfloor )$ in \cite{OSpHarm} if $\lfloor m/2\rfloor$ is even.
\end{remark}

From the fact that the spinor representations are completely pointed we obtain immediately that the tensor products with finite dimensional representations have bounded multiplicities. For the case $\mathfrak{sp}(2n)$ all infinite dimensional representations with bounded multiplicities can be obtained in this way, see \cite{MR1615943}.
\begin{corollary}
\label{bounded}
Assume the tensor product $\mS_{m|2n}\otimes L_\Lambda^{m|2n}$, for $L_\Lambda^{m|2n}$ a certain irreducible finite dimensional $\osp$-representation, is completely reducible. Each irreducible representation appearing in this decomposition is infinite dimensional and has bounded dimensions of its weight spaces. An upper bound for the dimension of the weight spaces is given by $dim_\mC\left( L_\Lambda^{m|2n}\right)$.
\end{corollary}

The realizations of $\mathfrak{osp}(m|2n)$ in this section correspond to the so-called para-boson and para-fermion statistics, where the bosonic and fermionic fields mutually anti-commute, see e.g. \cite{MR0660015}.

\section{Restrictions on the primitive vectors}
\label{sectionrestr}

We want to study the tensor product of the spinor representations $\mS^{(\pm)}_{m|2n}$ with finite dimensional $\osp$-highest weight representations. In \cite{MR1297597} it was proven that the tensor product of the $\mathfrak{sp}(2n)$-spinors with finite dimensional representations are always completely reducible. Since the spinors for $\mathfrak{so}(m)$ are finite dimensional the tensor product with finite dimensional representations will also always be completely reducible. In the super case the spinor representation decomposes into the spinors for $\mathfrak{so}(m)$ and $\mathfrak{sp}(2n)$ as an $\mathfrak{so}(m)\oplus\mathfrak{sp}(2n)$-representation (see equation \eqref{decompSp}), while irreducible finite dimensional $\osp$-representations also decompose into irreducible finite dimensional $\mathfrak{so}(m)\oplus \mathfrak{sp}(2n)$-representaions. This implies that for $\osp$ the tensor product of the spinors with a finite dimensional representation will always be completely reducible as an $\mathfrak{so}(m)\oplus\mathfrak{sp}(2n)$-representation.

In this section, as a first step in the calculation of the decomposition of the tensor products we find restrictions on the weights and multiplicities for primitive vectors.

\begin{lemma}
\label{highestnotzero}
Consider the tensor products $\mS_{2d+1|2n}\otimes K^{2d+1|2n}_\Lambda$, $\mS^+_{2d|2n}\otimes K^{2d|2n}_\Lambda$ and $\mS^-_{2d|2n}\otimes K^{2d|2n}_\Lambda$, with $K^{m|2n}_\Lambda$ an irreducible finite dimensional highest weight representation of $\mathfrak{osp}(m|2n)$. If 
\begin{eqnarray}
\label{exprmax}
w^+&=&\sum_{k=1}^{\dim  K^{m|2n}_\Lambda} p_k\otimes v_k
\end{eqnarray}
(with $v_k$ a $\mC$-basis for $K^{m|2n}_\Lambda$ and $v_1$ the vector with weight $\Lambda$) is a nonzero primitive vector in this tensor product, then the element $p_1\in\mS_{m|2n}$ is not zero. This implies that there can be at most one primitive vector of a certain weight.
\end{lemma}
\begin{proof}
Consider a maximal vector $w^+$ in $\mS^{(\pm)}_{m|2n}\otimes K^{m|2n}_\Lambda$, which we can assume to be a weight vector. Equation \eqref{exprmax} can be rewritten as 
\begin{eqnarray*}
w^+&=&\sum_{i=p}^{N} q_i\otimes v^{(i)}
\end{eqnarray*}
with $N$ the amount of different weight spaces appearing in $K^{m|2n}_\Lambda$, $q_i$ elements of the basis of monomials in equation \eqref{weightmonomials}, such that weight$(q_i)<$ weight$(q_{i+1})$ and $v^{(i)}$ the unique (since $\mS^{(\pm)}_{m|2n}$ is completely pointed) element $K^{m|2n}_\Lambda$ that appears in the summation with weight equal to weight$(w^+)-$weight$(q_i)$, with $N\ge p\ge 1$ and $v^{(p)}\not=0$. Then we find 
\begin{eqnarray*}
X_\alpha w^+&=&\sum_{i=p}^{N} (X_{\alpha}q_i)\otimes v^{(i)}+\sum_{i=p}^{N} (-1)^{|X_\alpha||q_i|}q_i\otimes (X_{\alpha}v^{(i)})
\end{eqnarray*}
for each positive root $\alpha$. Since the term $q_p\otimes (X_{\alpha}v^{(p)})$ must be zero we find that $v^{(p)}$ is the maximal vector of $K^{m|2n}_\Lambda$, so $p=1$ and the first part of the lemma is proven.

If there were two primitive vectors in $\mS^{(\pm)}_{m|2n}\otimes K^{m|2n}_\Lambda$ with the same weight, adding them up with suitable constants would yield a primitive vector without $v_1$. This is impossible by the first part of the lemma.
\end{proof}

Using this lemma we can restrict the possible weights of primitive vectors that will appear inside the tensor product $\mS_{m|2n}\otimes L^{m|2n}_\Lambda$, again we assume that $k_{d}=0$ implies $k_{d-1}=0$ for $\mathfrak{so}(2d)$.

\begin{theorem}
\label{posmax}
Consider the tensor product $\mS_{m|2n}\otimes K^{m|2n}_\Lambda$, with $d=\lfloor m/2\rfloor$ and $K^{m|2n}_\Lambda$ a finite dimensional irreducible representation with highest weight
\begin{eqnarray*}
\Lambda=\lambda+\kappa&=&\sum_{j=1}^dk_j\epsilon_j+\sum_{i=1}^nt_i\delta_i\qquad\mbox{(where $t_{k_d+1}=0$ must hold if $k_d<n$ by consistency)}\\
&=&\sum_{j=1}^d\lambda_j \omega_j+\sum_{i=1}^n\kappa_j\nu_j.
\end{eqnarray*}
If $w^+$ is a nonzero primitive vector in this tensor product then $w^+$ must have a weight of the form
\[
\Lambda-\mu-\rho+\omega_d-\frac{1}{2}\nu_n\]
with $\mu\in I_\lambda$ given in equation \eqref{setweightsclass} and $\rho=\sum_{i=1}^n l_i\delta_i$ satsfying $0\le l_i\le \kappa_i$ for $1\le i< n$ and 
\[l_n\le \begin{cases}
  2\kappa_n &\mbox{if }m=2d+1\\
  2\kappa_n+1&\mbox{if }m=2d.
 \end{cases}
 \]
\end{theorem}
\begin{proof}
First we consider the case $m=2d+1$. The vector $w^+$ is of the form $q_0\otimes v_1+\sum_{l=2}^{\dim K^{2d+1|2n}_\Lambda}p_l\otimes v_l$, with $v_1$ the highest weight vector of $K^{2d+1|2n}_\Lambda$. The element $q_0\in\Lambda_{d|n}$ is not zero because of Lemma \ref{highestnotzero}. For each $k=1,\cdots,d+n$, we define
\begin{eqnarray*}
w^+_k&=&q_0\otimes v_1+\sum_{s=1}^\infty q_s^{(k)}\otimes(Y_{\alpha_k})^sv_1
\end{eqnarray*}
as a part in the summation corresponding to $w^+$. The summations are in fact finite. From
\begin{eqnarray*}
X_{\alpha_j}\left(Y_{\alpha_j}\right)^lv_1&=&l(\lambda_j-l+1)\left(Y_{\alpha_j}\right)^{l-1}v_1\qquad\mbox{for}\quad 1\le j<d\\
X_{\alpha_{d+i}}\left(Y_{\alpha_{d+i}}\right)^lv_1&=&l(\kappa_i-l+1)\left(Y_{\alpha_{d+i}}\right)^{l-1}v_1\qquad\mbox{for}\quad 1\le i<n\\
X_{\alpha_{d+n}}\left(Y_{\alpha_{d+n}}\right)^{2l}v_1&=&-l\left(Y_{\alpha_{d+n}}\right)^{2l-1}v_1\quad\mbox{ and }\quad X_{\alpha_{d+n}}\left(Y_{\alpha_{d+n}}\right)^{2l+1}v_1=(\kappa_n-l)\left(Y_{\alpha_{d+i}}\right)^{2l}v_1,
\end{eqnarray*}
the fact that $K^{m|2n}_\Lambda$ is irreducible and $Y_{\alpha_d}^2=\frac{1}{2}[Y_{\alpha_d},Y_{\alpha_d}]=0$ we find the following results:
\begin{eqnarray*}
\left(Y_{\alpha_j}\right)^{\lambda_j+1}v_1=0\mbox{ for }1\le j< d && \left(Y_{\alpha_d}\right)^{2}v_1=0\\
\left(Y_{\alpha_{d+i}}\right)^{\kappa_i+1}v_1=0\mbox{ for }1\le i< n && \left(Y_{\alpha_{d+n}}\right)^{2\kappa_n+1}v_1=0.
\end{eqnarray*}
Therefore we can restrict the summation in $w^+_k$. Now because of the form of $w^+_k$, $X_{\alpha_k}w^+=0$ implies $X_{\alpha_k}w^+_k=0$.

We know that $q_0= \theta_1^{\gamma_1}\theta_2^{\gamma_2}\cdots\theta_d^{\gamma_d}t_1^{\beta_1}\cdots t_n^{\beta_n}$ with $\gamma_j\in\{0,1\}$ and $\beta_i\in\mN$. Expressing $X_{\alpha_j}w_j^+=0$ for $1\le j< d$ leads to the equation
\begin{eqnarray*}
\theta_{d-j}\partial_{\theta_{d-j+1}}q_0&=&-\lambda_jq_1^{(j)}.
\end{eqnarray*}
Therefore if $\lambda_j=0$ for $j<d$, this implies $\gamma_{d-j+1}=0$ or $\gamma_{d-j}=1$, which is equivalent with $\gamma_{d-j+1}-\gamma_{d-j}\le 0$. The same reasoning for $X_{\alpha_d}$ leads to the result that $k_d+l_1=0$ implies $\gamma_1=0$, which can be simplified, since $k_d=0$ implies $l_1=0$. These condition can be rewritten as
\begin{eqnarray}
\label{condgamma}
\gamma_{d-j+1}-\gamma_{d-j}\le \lambda_j\quad\mbox{for}\quad 1\le j\le d-1&\mbox{and}&2\gamma_1\le \lambda_d.
\end{eqnarray}

Considering $k>d$ but $k<d+n$ yields $q_j^{(k)}\sim (t_{n-i}\partial_{t_{n-1+1}})^jq_0$, which implies the condition $\partial_{t_{n-i+1}}^{\kappa_i+1}q_0=0$ holds, so we find that $\beta_{n-i+1}\le \kappa_i$ holds for $1\le i <n$. The case $d+n$ then implies $\beta_1\le 2\kappa_n$.

Equation \eqref{weightmonomials} shows that the weight of $q_0$ is given by 
\[-(\gamma_d\epsilon_1+\gamma_{d-1}\epsilon_2+\cdots+\gamma_1\epsilon_d)-(\beta_n\delta_1+\beta_{n-1}\delta_2+\cdots+\beta_1\delta_n)+\omega_d-\frac{1}{2}\nu_n,\]
which implies the condition obtained on $\{\gamma_j\}$ in equation \eqref{condgamma} are the same as on the corresponding weight $(\gamma_d\epsilon_1+\gamma_{d-1}\epsilon_2+\cdots+\gamma_1\epsilon_d)$ given in equation \eqref{setweightsclass} and the same statement holds for the $\beta_i$.

The reasoning remains identical for the case $m=2d$ only the approach to the last root vector $Y_{\alpha_{d+n}}$ changes slightly.
\end{proof}

In case $m=2d$ it can be seen immediately which weights correspond to primitive vectors in $\mS^+_{2d|2n}\otimes K^{2d|2n}_\Lambda$ and in $\mS^-_{2d|2n}\otimes K^{2d|2n}_\Lambda$, since the weight $\mu+\rho$ must appear in the representation $\mS^+_{2d|2n}$ or $\mS^-_{2d|2n}$.

\begin{remark}
\label{whysecondroot}
It would be more difficult to obtain a good restriction on the primitive vectors using the standard root system. This is a consequence of the fact that $\beta_n=\delta_n-\epsilon_1$ appears as a positive simple root. The corresponding positive root vector $\theta_d\partial_{t_1}$ is nilpotent while it should give information on the number of times the not nilpotent variable $t_1$ can appear.
\end{remark}


\section{Decomposition of the tensor products}
\label{secDec}

In this section we will calculate the actual decomposition of the tensor products $\mS_{m|2n}\otimes L^{m|2n}_\Lambda$ for a large class of highest weights $\Lambda$. We will state the results both in the standard choice of positive roots and in our choice, see Section \ref{sectionroot}. All calculations and proofs will be done in our choice.

First we focus on the case $\mS_{m|2n}\otimes K_{k\epsilon_1}^{m|2n}$.
\begin{theorem}
\label{decomp2d}
For $j,l\in\mN$ with $1\le j\le n$ and $1\le l$ the tensor product decompositions of $D(d|n)=\mathfrak{osp}(2d|2n)$-representations 
\begin{eqnarray*}
L_{\omega_d-\frac{1}{2}\nu_n}^{2d|2n}\otimes L_{\nu_j}^{2d|2n}&=&L_{\nu_j+\omega_d-\frac{1}{2}\nu_n}^{2d|2n}\oplus L_{\nu_{j-1}+\omega_{d-1}-\frac{1}{2}\nu_n}^{2d|2n}\quad\mbox{if}\quad j+d\not=n+1\\
L_{\omega_d-\frac{1}{2}\nu_n}^{2d|2n}\otimes L_{l\epsilon_1+\nu_n}^{2d|2n}&=&L_{l\epsilon_1+\omega_d+\frac{1}{2}\nu_n}^{2d|2n}\oplus L_{(l-1)\epsilon_1+\omega_{d-1}+\frac{1}{2}\nu_n}^{2d|2n}\\
L_{\omega_{d-1}-\frac{1}{2}\nu_n}^{2d|2n}\otimes L_{\nu_j}^{2d|2n}&=&L_{\nu_j+\omega_{d-1}-\frac{1}{2}\nu_n}^{2d|2n}\oplus L_{\nu_{j-1}+\omega_d-\frac{1}{2}\nu_n}^{2d|2n}\quad\mbox{if}\quad  j+d\not=n+1\\
L_{\omega_{d-1}-\frac{1}{2}\nu_n}^{2d|2n}\otimes L_{l\epsilon_1+\nu_n}^{2d|2n}&=&L_{l\epsilon_1+\omega_{d-1}+\frac{1}{2}\nu_n}^{2d|2n}\oplus L_{(l-1)\epsilon_1+\omega_d+\frac{1}{2}\nu_n}^{2d|2n},
\end{eqnarray*}
hold. For $1\le k\in\mN$ and $C(n+1)=\mathfrak{osp}(2|2n)$ the decompositions into irreducible representations are
\begin{eqnarray*}
K_{\frac{1}{2}\epsilon-\frac{1}{2}\nu_n}^{2|2n}\otimes K_{k\epsilon}^{2|2n}&=&K_{(k+\frac{1}{2})\epsilon-\frac{1}{2}\nu_n}^{2|2n}\oplus K_{(k-\frac{1}{2})\epsilon+\nu_{n-1}-\frac{3}{2}\nu_n}^{2|2n}\\
K_{\frac{1}{2}\epsilon+\nu_{n-1}-\frac{3}{2}\nu_n}^{2|2n}\otimes K_{k\epsilon}^{2|2n}&=&K_{(k+\frac{1}{2})\epsilon+\nu_{n-1}-\frac{3}{2}\nu_n}^{2|2n}\oplus K_{(k-\frac{1}{2})\epsilon-\frac{1}{2}\nu_n}^{2|2n},
\end{eqnarray*}
if $k\not=n$. For $j,l\in\mN$ with $1\le j\le n$ and $1\le l$ the tensor product decompositions of $B(d|n)=\mathfrak{osp}(2d+1|2n)$-representations 
\begin{eqnarray*}
L_{\omega_d-\frac{1}{2}\nu_n}^{2d+1|2n}\otimes L_{\nu_j}^{2d+1|2n}&=&L_{\nu_j+\omega_d-\frac{1}{2}\nu_n}^{2d+1|2n}\oplus L_{\nu_{j-1}+\omega_{d}-\frac{1}{2}\nu_n}^{2d+1|2n}\\
L_{\omega_d-\frac{1}{2}\nu_n}^{2d+1|2n}\otimes L_{l\epsilon_1+\nu_n}^{2d+1|2n}&=&L_{l\epsilon_1+\omega_d+\frac{1}{2}\nu_n}^{2d+1|2n}\oplus L_{(l-1)\epsilon_1+\omega_{d}+\frac{1}{2}\nu_n}^{2d+1|2n},
\end{eqnarray*}
hold for $d>0$. For $j\in\mN$ with $1\le j <n$, the tensor product decompositions of $B(0|n)=\mathfrak{osp}(1|2n)$-representations
\begin{eqnarray*}
L_{-\frac{1}{2}\nu_n}^{1|2n}\otimes L_{\nu_j}^{1|2n}&=&L_{\nu_j-\frac{1}{2}\nu_n}^{1|2n}\oplus L_{\nu_{j-1}-\frac{1}{2}\nu_n}^{1|2n}\\
L_{-\frac{1}{2}\nu_n}^{1|2n}\otimes L_{\nu_n}^{1|2n}&=&L_{\frac{1}{2}\nu_n}^{1|2n}\oplus L_{\frac{1}{2}\nu_n-\delta_n}^{1|2n}\oplus L_{\frac{1}{2}\nu_n-2\delta_n}^{1|2n},
\end{eqnarray*}
hold. 
\end{theorem}
\begin{proof}
First we assume $d>0$. Applying the technique of odd reflections and equation \eqref{KLold} show that the statements in the theorem are equivalent with
\begin{eqnarray}
\label{2ndchoiceeasyprod}
K_{\omega_d-\frac{1}{2}\nu_n}^{2d|2n}\otimes K_{k\epsilon_1}^{2d|2n}&=&K_{k\epsilon_1+\omega_d-\frac{1}{2}\nu_n}^{2d|2n}\oplus K_{(k-1)\epsilon_1+\omega_d+\nu_{n-1}-\frac{3}{2}\nu_n}^{2d|2n}\qquad\mbox{if}\quad k+d\not=n+1\\
\label{2ndchoiceeasyprodb}
K_{\omega_d+\nu_{n-1}-\frac{3}{2}\nu_n}^{2d|2n}\otimes K_{k\epsilon_1}^{2d|2n}&=&K_{k\epsilon_1+\omega_d+\nu_{n-1}-\frac{3}{2}\nu_n}^{2d|2n}\oplus K_{(k-1)\epsilon_1+\omega_d-\frac{1}{2}\nu_n}^{2d|2n}\qquad\mbox{if}\quad k+d\not=n+1\\
\label{2ndchoiceeasyprodc}
K_{\omega_d-\frac{1}{2}\nu_n}^{2d+1|2n}\otimes K_{k\epsilon_1}^{2d+1|2n}&=&K_{k\epsilon_1+\omega_d-\frac{1}{2}\nu_n}^{2d+1|2n}\oplus K_{(k-1)\epsilon_1+\omega_d-\frac{1}{2}\nu_n}^{2d+1|2n}.
\end{eqnarray}

According to Theorem \ref{posmax} the only two possible primitive vectors in the tensor product on the left-hand side of \eqref{2ndchoiceeasyprod} have weight $k\epsilon_1+\omega_d-\frac{1}{2}\nu_n$ and $(k-1)\epsilon_1+\omega_d+\nu_{n-1}-\frac{3}{2}\nu_n$. The first one obviously appears as $1\otimes v_1$. We explicitly prove that the second one also appears and show that it is not an element of $\cU(\mathfrak{osp}(2d|2n))\cdot (1\otimes v_1)$  when $k+d\not=n+1$. Equation \eqref{2ndchoiceeasyprod} then follows from Corollary \ref{ThmComRed}.

For $j=1,\cdots,d+n+1$ we define non-zero vectors
\begin{eqnarray*}
a_j&=&Y_{\alpha_j}Y_{\alpha_{j+1}}\cdots Y_{\alpha_{d+n}}1\in \mS_{2d|2n}^+\\
b_j&=&Y_{\alpha_{j-1}}Y_{\alpha_{j-2}}\cdots Y_{\alpha_1}v_1\in K_{k\epsilon_1}^{2d|2n}
\end{eqnarray*}
for $v_1$ the highest weight vector of $K_{k\epsilon_1}^{2d|2n}$.

A few calculations then yield the result
\[X_{\alpha_j}a_j=\begin{cases}
a_{j+1}&\mbox{if }j<d+n-1\\
2a_{j+1}&\mbox{if }j=d+n-1\\
-\frac{1}{2}a_{j+1}&\mbox{if }j=d+n
\end{cases}\] 
and
\[X_{\alpha_j}b_{j+1}=\begin{cases}
kb_{j}&\mbox{if }j=1\\
b_{j}&\mbox{if }1 < j<d\mbox{ or }d<j\le d+n\\
-b_j&\mbox{if }j=d.
\end{cases}\] 
It can also be checked that $X_{\alpha_k}a_j=0$ if $k\not=j$ and $X_{\alpha_k}b_{j}=0$ if $k\not=j+1$.

We define the vector $w\in K_{\omega_d-\frac{1}{2}\nu_n}^{2d|2n}\otimes K_{k\epsilon_1}^{2d|2n}$ as
\begin{eqnarray*}
w&=&a_1\otimes b_1+\frac{1}{k}\sum_{j=1}^{d-1}(-1)^j a_{j+1}\otimes b_{j+1}+\frac{(-1)^d}{k}\sum_{i=1}^{n-1}(-1)^ia_{d+i}\otimes b_{d+i}\\
&+&\frac{2(-1)^{d+n}}{k}a_{d+n}\otimes b_{d+n}+\frac{(-1)^{d+n}}{k}a_{d+n+1}\otimes b_{d+n+1}.
\end{eqnarray*}

The previous calculations then yield
\begin{eqnarray*}
X_{\alpha_1}w&=&a_2\otimes b_1+\frac{1}{k}(-1)ka_2\otimes b_1=0\\
kX_{\alpha_k}w&=&(-1)^{k-1}a_{k+1}\otimes b_{k}+ (-1)^ka_{k+1}\otimes b_k=0\quad\mbox{for } 1 < k<d\\
kX_{\alpha_d}w&=&(-1)^{d-1}a_{d+1}\otimes b_d +(-1)^{d+1} a_{d+1}\otimes (-b_{d})=0\\
kX_{\alpha_{d+i}}w&=&(-1)^{d+i}a_{d+i+1}\otimes b_{d+i}+(-1)^{d+i+1}a_{d+i+1}\otimes b_{d+i+1}=0\quad\mbox{for }1\le i <n-1\\
kX_{\alpha_{d+n-1}}w&=&(-1)^{d+n-1}2a_{d+n}\otimes b_{d+n-1}+2(-1)^{d+n}a_{d+n}\otimes b_{d+n-1}=0\\
kX_{\alpha_{d+n}}w&=&2(-1)^{d+n}(-1/2)a_{d+n+1}\otimes b_{d+n}+(-1)^{d+n}a_{d+n+1}\otimes b_{d+n}=0
\end{eqnarray*}
which shows that $w$ is the primitive vector of weight $(k-1)\epsilon_1+\omega_d+\nu_{n-1}-\frac{3}{2}\nu_n$.

Next we look at the vectors of weight $(k-1)\epsilon_1+\omega_d+\nu_{n-1}-\frac{3}{2}\nu_n$ inside $\cU(\mathfrak{osp}(2d|2n))\cdot (1\otimes v_1)$. These correspond to all the ways we can order $Y_{\alpha_1},\cdots,Y_{\alpha_{d+n}}$ such that the action on $1\otimes v_1$ does not give zero. The only vectors that can appear on the first spot are $Y_{\alpha_1}$ (which is non-zero on $v_1$) and $Y_{\alpha_{d+n}}$ (which is non-zero on $1\in \Lambda_{d|n}$). By continuing this we find the following possibilities:
\begin{eqnarray*}
Y_{\alpha_{j}}Y_{\alpha_{j+1}}\cdots Y_{\alpha_{d+n}}Y_{\alpha_{j-1}}Y_{\alpha_{j-2}}\cdots Y_{\alpha_1} (1\otimes v_1)&=&a_j\otimes b_j+a_{j+1}\otimes b_{j+1}
\end{eqnarray*}
for $1\le j\le d+n$.

A general element inside $K_{\frac{1}{2}\omega_d-\frac{1}{2}\nu_n}^{2d|2n}\otimes K_{k\epsilon_1}^{2d|2n}$ of weight $(k-1)\epsilon_1+\omega_d+\nu_{n-1}-\frac{3}{2}\nu_n$ is of the form
\begin{eqnarray*}
\sum_{j=1}^{d+n+1}C_{j}\,a_j\otimes b_j
\end{eqnarray*}
for arbitrary constants $C_j$. Such a vector is inside $\cU(\mathfrak{osp}(2d|2n))\cdot (1\otimes v_1)$ if and only if $\sum_{j=1}^{d+n+1}(-1)^j C_j=0$ holds according to the calculations above. A quick calculation therefore shows that $w\in\cU(\mathfrak{osp}(2d|2n))\cdot (1\otimes v_1)$ holds if and only if $k+d-n-1=0$ holds, which is exactly the case we excluded.

Equation \eqref{2ndchoiceeasyprodb} is proved using the same techniques, but the $a_j$ vectors are now derived from $t_1\in\mS^-_{2d|2n}$ and only the simple roots $\alpha_1,\cdots,\alpha_{d+n-1}$ play a role. The proof of equation \eqref{2ndchoiceeasyprodc} is very similar, but the calculations change slightly because there are two odd simple root vectors and more importantly now $X_{\alpha_{d+n-1}}a_{d+n-1}=a_{d+n}$ holds. This leads to the difference that the second primitive vector is never generated by the first for $\mathfrak{osp}(2d+1|2n)$, regardless of the value of $k$. Checking that $v_2^+\not\in\cU(\osp)\cdot v_1^+$ for these cases can also be done using the quadratic Casimir operator, as has been explained in Section \ref{criterion}.

The case $\mathfrak{osp}(1|2n)$ follows similarly.
\end{proof}
Note that for the case $C(n+1)=\mathfrak{osp}(2|2n)$, contrary to the case $D(d|n)$, Theorem \ref{decomp2d} does not have an analog for $\mathfrak{so}(2)$. The statements do not hold when we would substitute $n=0$.

Now we focus on the case excluded from Theorem \ref{decomp2d}. 
\begin{theorem}
\label{tensornotcr}
If $n\ge d$ the tensor products $\mS^\pm_{2d|2n}\otimes K^{2d|2n}_{(n-d+1)\epsilon_1}$ are indecomposable but not irreducible. The representation has subrepresentations
\begin{eqnarray*}
K^{2d|2n}_{\omega_d-\frac{1}{2}\nu_n}\otimes K^{2d|2n}_{(n-d+1)\epsilon_1}\supset V\supset K^{2d|2n}_{(n-d)\epsilon_1+\omega_d+\nu_{n-1}-\frac{3}{2}\nu_n},
\end{eqnarray*}
with $V$ an indecomposable representation satisfying
\begin{eqnarray*}
V/ K^{2d|2n}_{(n-d)\epsilon_1+\omega_d+\nu_{n-1}-\frac{3}{2}\nu_n}&\cong & K^{2d|2n}_{(n-d+1)\epsilon_1+\omega_d-\frac{1}{2}\nu_n}.
\end{eqnarray*}
For the other spinor space this is given by
\begin{eqnarray*}
K^{2d|2n}_{\omega_d+\nu_{n-1}-\frac{3}{2}\nu_n}\otimes K^{2d|2n}_{(n-d+1)\epsilon_1}\supset U\supset K^{2d|2n}_{(n-d)\epsilon_1+\omega_d-\frac{1}{2}\nu_n},
\end{eqnarray*}
with $U$ an indecomposable representation satisfying
\begin{eqnarray*}
U/ K^{2d|2n}_{(n-d)\epsilon_1+\omega_d-\frac{1}{2}\nu_n}&\cong & K^{2d|2n}_{(n-d+1)\epsilon_1+\omega_d+\nu_{n-1}-\frac{3}{2}\nu_n}.
\end{eqnarray*}
\end{theorem}
\begin{proof}
This follows immediately from Theorem \ref{Simplecasencr} and the proof of Theorem \ref{decomp2d}.
\end{proof}

\begin{remark}
In a forthcoming paper we will prove that the decomposition series in this paper is complete, i.e. that 
\begin{eqnarray*}
\left(K^{2d|2n}_{(n-d+1)\epsilon_1}\otimes K^{2d|2n}_{\omega_d-\frac{1}{2}\nu_n}\right)/V&\cong& K^{2d|2n}_{(n-d)\epsilon_1+\omega_d+\nu_{n-1}-\frac{3}{2}\nu_n}\mbox{ and}\\
\left(K^{2d|2n}_{(n-d+1)\epsilon_1}\otimes K^{2d|2n}_{\omega_d+\nu_{n-1}-\frac{3}{2}\nu_n}\right)/U&\cong& K^{2d|2n}_{(n-d)\epsilon_1+\omega_d-\frac{1}{2}\nu_n}
\end{eqnarray*}
hold. This will be done by constructing an explicit realization of these tensor products in an analytical theory.
\end{remark}

Theorem \ref{tensornotcr} showed that the tensor product $ \mS_{2d|2n}\otimes K^{2d|2n}_{k\epsilon_1}$ is not always completely reducible. Similarly the case $\mS_{m|2n}\otimes K^{m|2n}_{k\epsilon_1+l\epsilon_2}$ will sometimes not be completely reducible, even for $m$ odd. However, we can prove the following lemma.
\begin{lemma}
\label{classcr2}
The tensor product
\begin{eqnarray*}
 \mS_{2d+1|2n}\otimes K_{k\epsilon_1+l\epsilon_2}^{2d+1|2n}
\end{eqnarray*}
is completely reducible unless $k+l=2+2n-2d$ holds.
\end{lemma}
\begin{proof}
Theorem \ref{posmax} implies that the possible primitive vectors have weight 
\[k\epsilon_1+l\epsilon_2+\omega_d-\frac{1}{2}\nu_n,\,  k\epsilon_1+(l-1)\epsilon_2+\omega_d-\frac{1}{2}\nu_n,\, (k-1)\epsilon_1+(l-1)\epsilon_2+\omega_d-\frac{1}{2}\nu_n\]
and $(k-1)\epsilon_1+l\epsilon_2+\omega_d-\frac{1}{2}\nu_n$ if $l<k$. Denote the corresponding highest weights in the standard root system by $\kappa_1$, $\kappa_2$, $\kappa_3$ and $\kappa_4$. The eigenvalues of the Casimir on such primitive vectors can be calculated as explained in Section \ref{criterion}, which yields
\begin{eqnarray*}
\langle \kappa_1,\kappa_1+2\rho\rangle& =& \left(d-n-\frac{1}{2}\right)(k+l)+\frac{1}{2}\left(k(k+1)+l(l-1)\right)+\frac{1}{8}(2d+1-2n)(d-n).
\end{eqnarray*}
The difference between these values for the different weights then becomes
\begin{eqnarray*}
\langle \kappa_1,\kappa_1+2\rho\rangle-\langle \kappa_2,\kappa_2+2\rho\rangle=d-n+l-\frac{3}{2}&\not=&0\\
\langle \kappa_1,\kappa_1+2\rho\rangle-\langle \kappa_3,\kappa_3+2\rho\rangle=k+l+2d-2n-2&\not=&0\quad \mbox{if}\quad k+l\not=2n+2-2d\\
\langle \kappa_1,\kappa_1+2\rho\rangle-\langle \kappa_4,\kappa_4+2\rho\rangle=d-n+k-\frac{1}{2}&\not=&0\\
\langle \kappa_2,\kappa_2+2\rho\rangle-\langle \kappa_3,\kappa_3+2\rho\rangle=d-n+k-\frac{1}{2}&\not=&0\\
\langle \kappa_2,\kappa_2+2\rho\rangle-\langle \kappa_4,\kappa_4+2\rho\rangle=k-l+1&\not=&0\quad \mbox{if}\quad k>l\\
\langle \kappa_3,\kappa_3+2\rho\rangle-\langle \kappa_4,\kappa_4+2\rho\rangle=\frac{3}{2}-d+n-l&\not=&0.
\end{eqnarray*}
Therefore the condition $k+l\not=2n+2-2d$ is sufficient to conclude complete reducibility by using Corollary \ref{ThmComRed}.
\end{proof}
To obtain the actual decomposition of this tensor product the existence of the primitive vectors needs to be proven. This will be done in case $k,l>n$ as part of the result in Theorem \ref{finaldecomp2}.

Theorem \ref{decomp2d} implied that $ \mS_{m|2n}\otimes K^{m|2n}_{k\epsilon_1}$ is always completely reducible for $k>n$. Lemma shows that $\mS_{2d+1|2n}\otimes K_{k\epsilon_1+l\epsilon_2}^{2d+1|2n}$ is completely reducible as well if $k,l >n$. It turns out that this condition can be extended, the tensor product $\mS_{m|2n}\otimes K^{m|2n}_{\sum_{j=1}^al_j\epsilon_j}$ with $a\le d$ and $l_j >n$ will always be completely reducible. In the following theorems we determine the decomposition into irreducible representations explicitly.

\begin{theorem}
\label{finalthmdecomp}
Consider an irreducible finite dimensional highest weight $\mathfrak{osp}(m|2n)$-representation
\begin{eqnarray*}
L_{\Lambda}^{m|2n}&\mbox{with}& \Lambda=\sum_{j=1}^a k_j\epsilon_j+a\nu_n \qquad\mbox{where}\quad a\le d \mbox{ and }k_j\ge 1,
\end{eqnarray*}
with $d=\lfloor m/2\rfloor$ and where $a=d-1$ is not allowed for $m=2d$. The decomposition
\begin{eqnarray*}
L^{2d+1|2n}_{\omega_d-\frac{1}{2}\nu_n}\otimes L^{2d+1|2n}_{\Lambda}&=&\bigoplus_{\kappa\in I_\lambda}  L^{2d+1|2n}_{\Lambda-\kappa+\omega_d-\frac{1}{2}\nu_n}
\end{eqnarray*}
holds for $\lambda=\sum_{j=1}^a k_j\epsilon_j$ and $I_\lambda$ given in equation \eqref{setweightsclass}. The decompositions
\begin{eqnarray*}
L^{2d|2n}_{\omega_d-\frac{1}{2}\nu_n}\otimes L^{2d|2n}_{\Lambda}&=&\bigoplus_{\kappa\in I_\lambda}  L^{2d|2n}_{\Lambda-\kappa-\sigma(\kappa)\epsilon_d+\omega_d-\frac{1}{2}\nu_n}\\
L^{2d|2n}_{\omega_{d-1}-\frac{1}{2}\nu_n}\otimes L^{2d|2n}_{\Lambda}&=&\bigoplus_{\kappa\in I_\lambda}  L^{2d|2n}_{\Lambda-\kappa+\sigma(\kappa)\epsilon_{d}+\omega_{d-1}-\frac{1}{2}\nu_n}
\end{eqnarray*}
with $\sigma\left(\sum_j i_j\epsilon_j\right)$ equal to $0$ (respectively $1$) if an even (respectively odd) number of $i_j$ is non-zero.
\end{theorem}
\begin{proof}
First we consider the case $m=2d+1$. By Theorem \ref{changeroot}, in our choice of root system the tensor product is given by $K^{2d+1|2n}_{\omega_d-\frac{1}{2}\nu_n}\otimes K^{2d+1|2n}_{\sum_{j=1}^a(k_j+n)\epsilon_j}$. Theorem \ref{posmax} implies that the weights of the possible primitive vectors are given by $\sum_{j=1}^a(k_j+n)\epsilon_j-\kappa+\omega_d-\frac{1}{2}\nu_n$ with $\kappa\in I_{\sum_{j=1}^a(k_j+n)\epsilon_j}=I_\lambda$. Even without complete reducibility it is possible to conclude that the possible weights of primitive vectors in the standard root system are given by the technique of odd reflections, the resulting weights are
\[\sum_{j=1}^ak_j\epsilon_j-\kappa+\omega_d+a\nu_n-\frac{1}{2}\nu_n.\]
Now if $y^+\in\cU(\mg)\cdot x^+$ for two such primitive vectors $x^+$ and $y^+$, with $\mg=\mathfrak{osp}(2d+1|2n)$, there is an element $g\in \cU(\mathfrak{n}^-)$ such that $y^+=gx^+$. However, the difference in the weights between two primitive vectors is always an $\mathfrak{so}(2d+1)$-weight. The structure of the positive simple roots for the standard root system (the $\beta_k$'s in Section \ref{sectionroot}) implies that $g\in\cU(\mathfrak{so}(2d+1))$. Since the tensor product is completely reducible as an $\mathfrak{so}(2d+1)\oplus\mathfrak{sp}(2n)$-representation this is impossible.

Corollary \ref{ThmComRed} then implies that the tensor product is completely reducible, so
\begin{eqnarray}
\label{partdecomp}
L^{2d+1|2n}_{\omega_d-\frac{1}{2}\nu_n}\otimes L^{2d+1|2n}_{\Lambda}&=&\bigoplus_{\kappa\in I_\lambda^\ast}  L^{2d+1|2n}_{\Lambda-\kappa+\omega_d-\frac{1}{2}\nu_n}
\end{eqnarray}
holds for some subset $I_\lambda^\ast\subset I_\lambda$. 

Now we prove that $I_\lambda^\ast=I_\lambda$. For each $\kappa\in I_\lambda$ there is a vector 
\[w_0^+(\kappa)\in \left(L^{2d+1|0}_{\omega_d}\otimes L^{2d+1|0}_{\lambda}\right) \times\left(L^{0|2n}_{-\frac{1}{2}\nu_n}\otimes L^{0|2n}_{a\nu_n}\right) \subset L^{2d+1|2n}_{\omega_d-\frac{1}{2}\nu_n}\otimes L^{2d+1|2n}_{\Lambda},\]
of weight $\sum_{j=1}^ak_j\epsilon_j-\kappa+\omega_d +a\nu_n-\frac{1}{2}\nu_n$, corresponding to the $\mathfrak{so}(2d+1)$-maximal vector in $L^{2d+1|0}_{\omega_d}\otimes L^{2d+1|0}_{\lambda}$ of weight $\sum_{j=1}^ak_j\epsilon_j-\kappa+\omega_d$ in Theorem \ref{classdecomp} and the highest weight vector of $L^{0|2n}_{-\frac{1}{2}\nu_n}\otimes L^{0|2n}_{a\nu_n}$. This is a maximal vector for $\mathfrak{so}(2d+1)\oplus\mathfrak{sp}(2n)$. We can prove that this vector can not be inside an irreducible representation in equation \eqref{partdecomp} generated by a maximal vector with a weight different from $\sum_{j=1}^ak_j\epsilon_j-\kappa+\omega_d +a\nu_n-\frac{1}{2}\nu_n$. This is again a consequence of the fact that the difference of weights is always an $\mathfrak{so}(2d+1)$-weight and the complete reducibility. So for each $\kappa\in I_\lambda$ there is a vector which is not generated by $\mathfrak{n}^-$-action. Theorem \ref{dimensequal} then implies that the proposed decomposition holds.

The proof for $m=2d$ is similar.
\end{proof}

In our choice of root system, Theorem \ref{finalthmdecomp} is rewritten as follows.
\begin{theorem}
\label{finaldecomp2}
Consider $\mu$ an integral dominant $\mathfrak{so}(m)$-weight of the form $\mu=\sum_{j=1}^a(k_j+n)\epsilon_j$ with $k_j\ge 1$ integers and $a\le d$ and $m=2d$ implies $a\not=d-1$. The following decomposition into irreducible $\mathfrak{osp}(m|2n)$-representations holds:
\begin{eqnarray*}
K^{2d+1|2n}_{\omega_d-\frac{1}{2}\nu_n}\otimes K^{2d+1|2n}_{\mu}&=&\bigoplus_{\kappa \in I_\mu}  K^{2d+1|2n}_{\mu-\kappa+\omega_d-\frac{1}{2}\nu_n}\\
K^{2d|2n}_{\omega_d-\frac{1}{2}\nu_n}\otimes K^{2d|2n}_{\mu}&=&\bigoplus_{\kappa \in I_\mu}  K^{2d|2n}_{\mu-\kappa+\omega_d-\frac{1}{2}\nu_n-\sigma(\kappa)\delta_n}\\
K^{2d|2n}_{\omega_d+\nu_{n-1}-\frac{3}{2}\nu_n}\otimes K^{2d|2n}_{\mu}&=&\bigoplus_{\kappa \in I_\mu}  K^{2d|2n}_{\mu-\kappa+\omega_d+\nu_{n-1}-\frac{3}{2}\nu_n+\sigma(\kappa)\delta_n}
\end{eqnarray*}
with $I_\mu$ given in equation \eqref{setweightsclass} and $\sigma(\kappa)$ as in Theorem \ref{finalthmdecomp}.
\end{theorem}

For the case $B(0|n)=\mathfrak{osp}(1|2n)$, theorems \ref{finalthmdecomp} and \ref{finaldecomp2} are clearly empty.

\section{Conclusion}
\label{concl}
The main results of this paper are the classification and realization of the completely pointed $\osp$-modules in Theorem \ref{supercompp}, Definition \ref{Xspin1} and Definition \ref{Xspin2} and the decomposition of the tensor products (which are e.g. useful for invariant differential operators in superspace) in Theorems \ref{decomp2d}, \ref{tensornotcr} and \ref{finalthmdecomp}. The results on completely pointed modules are generalizations of the results for $\mathfrak{sp}(2n)$ in \cite{MR1297597}, while the tensor products are mainly generalizations of results for $\mathfrak{so}(m)$. To obtain these results the insights on tensor product representations of semisimple Lie superalgebras in Theorem \ref{dimensequal} and Corollary \ref{ThmComRed} were very important. These considerations also give insight into the cases where the tensor product is not completely reducible, see Theorem \ref{Simplecasencr}.

Most of the tensor products we studied were completely reducible. Calculations such as Lemma \ref{classcr2} suggest that there are more irreducible highest weights that lead to a completely reducible representation. To obtain these, Theorem \ref{posmax} and Corollary \ref{ThmComRed} will be usefull, but the reasoning in Theorem \ref{finalthmdecomp} can no longer be used. Also the cases that are not completely reducible are interesting. Theorem \ref{Simplecasencr} for arbitrary representations of arbitrary Lie superalgebras is not complete in the sense that no statement is derived on the irreducibility of the representation $P/V$. For the particular case studied in Theorem \ref{tensornotcr} we will obtain in a forthcoming paper that $P/V$ is in fact irreducible as a side-result of an application of the tensor products. It is an interesting question whether this is a general property.

Other similar interesting representations to study would be the higher spinor representations for $\osp$ (which appear in the tensor products studied in this paper) and their tensor product with the fundamental representation $L_{\delta_1}^{m|2n}$. This would be a generalization to $\osp$ of the results in \cite{MR2286881} for $\mathfrak{sp}(2n)$ or in \cite{MR0482879} for $\mathfrak{so}(m)$.

Using the results of tensor product decompositions, the classification of $\mathfrak{osp}(m|2n)$-representations with bounded weight-multiplicities can then be addressed. As noted in corollary \ref{bounded}, the representations that appear in the tensor products have bounded weight-multiplicities. Special attention needs to be considered for the case where the tensor product is not completely reducible.
 The classical results for $\mathfrak{sp}(2n)$ are given in \cite{MR1615943}. 
 
The results in Theorem \ref{decomp2d} give the necessary representation-theoretical background to construct the super Dirac operator along the lines of the classical case in \cite{MR0223492} as well as a description of the kernel as an $\osp$-representation. This will be studied explicitly in a forthcoming paper.
\subsection*{Acknowledgment}
The author would like to thank Joris Van der Jeugt, Ruibin Zhang and Vladimir Soucek for many interesting remarks and discussions.

The author thanks Tom Ferguson and Dimitar Grantcharov for pointing out that Theorem 8 in the previous version did not cover the case $\mathfrak{osp}(1|2n)$.



\begin{thebibliography}{ASM}


\bibitem{MR1632811}
G. Benkart, C.L. Shader, A. Ram,
\newblock{Tensor product representations for orthosymplectic Lie superalgebras,}
\newblock{J. Pure Appl. Algebra} {130} (1998), 1--48. 

\bibitem{BG}
J.N. Bernstein, S.I. Gelfand,
\newblock{Tensor products of finite and infinite dimensional representations of semisimple Lie algebras,}
\newblock{Comp. Math., Vol 41, (1980), 245--285}

\bibitem{MR2782791}
F. Brackx, D. Eelbode, L. Van de Voorde,
\newblock{Higher spin Dirac operators between spaces of simplicial monogenics in two vector variables,}
\newblock Math. Phys. Anal. Geom. 14 (2011), 1--20. 


\bibitem{MR1297597}
D.J. Britten, J. Hooper, F.W. Lemire, 
\newblock{Simple $C_n$ modules with multiplicities 1 and applications,}
\newblock{Canad. J. Phys.} {72} (1994), 326--335. 

\bibitem{MR1615943}
D.J. Britten, F.W. Lemire, 
\newblock{On modules of bounded multiplicities for the symplectic algebras,}
\newblock{Trans. Amer. Math. Soc. 351 (1999), 3413--3431. }

\bibitem{MR2028498}
S.J. Cheng, R.B. Zhang, 
\newblock{Howe duality and combinatorial character formula for orthosymplectic Lie superalgebras},
\newblock{Adv. Math.} {182} (2004), 124--172.

\bibitem{OSpHarm}
{K. Coulembier},
\newblock{The orthosymplectic supergroup in harmonic analysis,}
\newblock{J. Lie Theory 23 (2013) 55-83.}

\bibitem{Joseph}
K. Coulembier, P. Somberg, V. Soucek,
\newblock{Joseph ideals for $\osp$,}
\newblock{Int. Math. Res. Not. 2013; doi: 10.1093/imrn/rnt074.}

\bibitem{MR0482879}
H.D. Fegan,
\newblock{Conformally invariant first order differential operators,}
\newblock{Quart. J. Math. Oxford (2) 27 (1976), 371--378. }

\bibitem{MR1013330}
S.L. Fernando,
\newblock{Lie algebra modules with finite-dimensional weight spaces. I,}
\newblock{Trans. Amer. Math. Soc. 322 (1990), 757--781. }

\bibitem{MR1773773}
{L. Frappat, A. Sciarrino, P. Sorba},
\newblock {\it Dictionary on {L}ie algebras and superalgebras},
\newblock Academic Press Inc., San Diego, CA, 2000.
  
\bibitem{MR051963}
{V. Kac,}
\newblock{Representations of classical Lie superalgebras}, 
\newblock{Lecture Notes in Math.} {676}, Springer, Berlin, 1978.

\bibitem{MR1327543}
V. Kac, M. Wakimoto,
\newblock{Integrable highest weight modules over affine superalgebras and number theory,}
\newblock{Lie theory and geometry, 415--456, Progr. Math., 123, BirkhŠuser Boston, Boston, MA, 1994. }

\bibitem{MR1738448}
A. Kor\'anyi, H.M. Reimann,
\newblock{Equivariant first order differential operators on boundaries of symmetric spaces,}
\newblock{Invent. Math. 139 (2000), 371--390.}

\bibitem{MR0400304}
B. Kostant,
\newblock{\it Symplectic spinors,}
\newblock{ Symposia Mathematica, Vol. XIV, pp. 139--152.}
Academic Press, London, 1974. 

\bibitem{MR2286881}
S. Kr\'ysl,
\newblock{Decomposition of a tensor product of a higher symplectic spinor module and the defining representation of sp(2n,C),}
J. Lie Theory 17 (2007), no. 1, 63--72. 

\bibitem{MR1037401}
K. Nishiyama,
\newblock{Oscillator representations for orthosymplectic algebras,}
\newblock J. Algebra 129 (1990), no. 1, 231--262. 

\bibitem{MR1401053}
O. Mathieu,
\newblock{On the dimension of some modular irreducible representations of the symmetric group,}
\newblock{Lett. Math. Phys. 38 (1996), no. 1, 23--32. }

\bibitem{MR0660015}
T. Palev, 
\newblock{Para-Bose and para-Fermi operators as generators of orthosymplectic Lie superalgebras,}
\newblock{J. Math. Phys. 23 (1982), no. 6, 1100--1102. }

\bibitem{MR1201236}
I. Penkov, V. Serganova,
\newblock{Representations of classical Lie superalgebras of type I,}
\newblock{Indag. Math. (N.S.) 3 (1992), no. 4, 419--466. }

\bibitem{MR0424886}
M. Scheunert, W. Nahm, V. Rittenberg,
\newblock{Graded Lie algebras: Generalization of Hermitian representations,}
\newblock J. Mathematical Phys. 18 (1977), 146--154. 

\bibitem{MR0223492}
E.M. Stein, G. Weiss,
\newblock{Generalization of the Cauchy-Riemann equations and representations of the rotation group,}
\newblock{Amer. J. Math. 90 (1968), 163--196. }

\bibitem{MR2395482}
R.B. Zhang, 
\newblock { Orthosymplectic Lie superalgebras in superspace analogues of quantum  Kepler problems},
\newblock {Comm. Math. Phys.} {280} (2008), 545--562.




\end{thebibliography}
\end{document}